\documentclass[11pt]{amsart}
\usepackage{amssymb,epic,eepic,verbatim,epsfig, url}
\usepackage{graphicx}
\numberwithin{subsection}{section} 

\newtheorem{theorem}{Theorem}[section]
\newtheorem{corollary}[theorem]{Corollary}

\newtheorem{proposition}[theorem]{Proposition}
\newtheorem{lemma}[theorem]{Lemma}
\newtheorem{lem}[theorem]{}
\theoremstyle{definition}
\newtheorem{definition}[theorem]{Definition}
\theoremstyle{remark}
\newtheorem{remark}[theorem]{Remark}
\newtheorem{example}[theorem]{Example}
\newcommand{\blem}{\begin{lem} \rm}
\newcommand{\elem}{\end{lem}}
\newcommand{\N}{\mathbb{N}}
\newcommand{\R}{\mathbb
{R}}
\renewcommand{\H}{\mathbb{H}}

\newcommand{\C}{\mathbb{C}}

\newcommand{\Z}{\mathbb{Z}}
\newcommand{\bA}{\mathbb{A}}

\renewcommand{\P}{\mathbb{P}}

\newcommand{\on}{\operatorname}
\renewcommand{\Vert}{\on{Vert}}
\newcommand{\dist}{\on{dist}}
\newcommand\bra[1]{ < \kern-.7ex {#1} \kern-.7ex >} 
\newcommand{\Edge}{\on{E}}

\newcommand{\Hom}{ \on{Hom}}

\newcommand{\ssm}{\kern-.5ex \smallsetminus \kern-.5ex}

\newcommand\dirac{/\kern-1.2ex\partial} 
\newcommand\qu{/\kern-.7ex/} 
\newcommand\lqu{\backslash \kern-.7ex \backslash} 

\newcommand\dr{r_+ \kern-.7ex - \kern-.7ex r_-}
 
\newcommand{\labell}\label

\newcommand{\lra}{\longrightarrow}

\newcommand{\ol}{\overline}

\newcommand\col{{\on{col}}}
\newcommand\WW{\on{W}}

\newcommand\ul{\underline}

\renewcommand\Im{\on{Im}}

\newcommand\bdefn{\begin{definition}}
\newcommand\edefn{\end{definition}}
\newcommand\bea{\begin{eqnarray*}}
\newcommand\eea{\end{eqnarray*}}
\newcommand\bcv{\left[ \begin{array}{r} }
\newcommand\ecv{\end{array} \right] }

\newcommand\bma{\left[ \begin{array} }
\newcommand\ema{\end{array} \right]}
\newcommand\ben{\begin{enumerate}}
\newcommand\een{\end{enumerate}}
\newcommand\bex{\begin{example}}
\newcommand\bsj{\left\{ \begin{array}{rrr} }
\newcommand\esj{\end{array} \right\}}

\newcommand\eex{\end{example}}

\newcommand\sx{*\kern-.5ex_X}
\newcommand{\ainfty}{{$A_\infty$\ }}

\def\mathunderaccent#1{\let\theaccent#1\mathpalette\putaccentunder}
\def\putaccentunder#1#2{\oalign{$#1#2$\crcr\hidewidth \vbox
to.2ex{\hbox{$#1\theaccent{}$}\vss}\hidewidth}}

\newcommand{\nN}{\ol{M}^w}

\begin{document}

\title[Geometric realizations of the multiplihedron]{Geometric
  realizations of the multiplihedron} \author{S. Mau}

\author{C. Woodward}

\begin{abstract}
We realize Stasheff's multiplihedron geometrically as the moduli space
of stable quilted disks.  This generalizes the geometric realization
of the associahedron as the moduli space of stable disks.  We show
that this moduli space is the non-negative real part of a complex
moduli space of {\em stable scaled marked curves}.
\end{abstract}

\maketitle

\section{Introduction}

The Stasheff polytopes, also known as {\em associahedra}, have had
many incarnations since their original appearance in Stasheff's work
on homotopy associativity \cite{stasheff}.  A particular realization
of the associahedra as the compactified moduli space of nodal disks
with markings is described by Fukaya and Oh \cite{zero-loop}. The
natural cell decomposition arising from this compactification is dual
to the cell decomposition arising from the compactification of a space
of metric trees studied by Boardman and Vogt \cite{boardman-vogt}.  In
this paper we describe analogous constructions for a related family of
polytopes $J_n$, called the {\em multiplihedra}, which appeared in
\cite{stasheff} when defining $A_\infty$ maps between $A_\infty$
spaces, see also Iwase and Mimura \cite{iwase-mimura}.  The
multiplihedra have a realization as metric trees with levels as found
in \cite{boardman-vogt}, which in a certain sense dualizes the CW
structure in Stasheff.  We consider a moduli space $M_{n,1}$ of marked
{\em quilted} disks, which are disks with $n+1$ marked points $z_0,
\ldots, z_n$ on the boundary, and an interior circle passing through
the marked point $z_0$.  This moduli space has a compactification
$\ol{M}_{n,1}$ by allowing nodal disks as in the definition of the
moduli space of stable marked disks.  Our first main result is

\begin{theorem} \label{main} The moduli space of stable $n$-marked quilted disk 
$\ol{M}_{n,1}$ is isomorphic as a CW-complex to the multiplihedron
  $J_n$.
\end{theorem}
\noindent Another geometric realization of the multiplihedron, which
gives a different CW structure, appears in Fukaya-Oh-Ohta-Ono
\cite{fooo}.  The authors of \cite{fooo} denote them by $\nN_{n}$ for
$k = 1,2,\ldots$ and use them to define \ainfty maps. The geometric
description of $\nN_{n}$ is similar to the space of quilted disks, in
that it is a moduli space of stable marked nodal disks with some
additional structure.  The main difference is that their complex is
has the structure of a manifold with corners, whereas the moduli space
of quilted disks has real toric singularities on its boundary.

Using our geometric realization, we introduce a natural
complexification of the multiplihedron.  The moduli space of quilted
disks ${M}_{n,1}$ can also be naturally identified with the moduli
space of $n$ points on the real line modulo translation only.  As
such, it sits inside the moduli space $M_{n,1}(\C)$ of $n$ points on
the complex plane modulo translation.  A natural compactification
$\ol{M}_{n,1}(\C)$ of this space was constructed in Ziltener's thesis
\cite{zilt:phd}, as the moduli space of symplectic vortices on the
affine line with trivial target.  Our second main result concerns the
structure of Ziltener's compactification $\ol{M}_{n,1}(\C)$, and its
relationship with the multiplihedron:

\begin{theorem} The moduli space of stable scaled marked curves 
$\ol{M}_{n,1}(\C)$ admits the structure of a complex projective
  variety with toric singularities that contains the multiplihedron
  $\ol{M}_{n,1}$ as a fundamental domain of the action of the
  symmetric group $S_n$ on its real locus.
\end{theorem}

This result is analogous to that for the Grothendieck-Knudsen moduli
space of genus zero marked stable curves, which contains the
associahedron as a fundamental domain for the action of the symmetric
group on its real locus.  In \cite{morphism} this moduli space is used
to define a notion of {\em morphism of cohomological field theories}.

\section{Background on associahedra}

Let $n > 2$ be an integer.  The $n$-th associahedron $K_n$ is a
$CW$-complex of dimension $n-2$ whose vertices correspond to the
possible ways of parenthesizing $n$ variables $x_1,\ldots,x_n$.  Each
facet of $K_n$ is the image of an embedding
\begin{equation} \label{boundary}
\phi_{i,e} : K_{i} \times K_{e} \to K_n, \ \ i + e = n+1 \end{equation} 
corresponding to the expression $x_1 \ldots x_{i-1} (x_i \ldots x_{i +
  e}) x_{i + e + 1} \ldots x_n $.  The associahedra have geometric
realizations as moduli spaces of genus zero nodal disks with markings:

\begin{definition} 
A {\em marked nodal disk} consists of a collection of disks, a
collection of {\em nodal points}, and a collection of {\em markings}
$(z_1,\ldots,z_n)$ disjoint from the nodes, in clockwise order around
the boundary, see \cite{zero-loop}.  The {\em combinatorial type} of
the nodal disk is the ribbon tree obtained by replacing each disk with
a vertex, each nodal point with a finite edge between the vertices
corresponding to the two disk components, and each marking with a
semi-infinite edge.  A marked nodal disk is {\em stable} if each disk
component contains at least three nodes or markings.  A {\em morphism}
between nodal disks is a collection of holomorphic isomorphisms
between the disk components, preserving the singularities and
markings.
\end{definition} 

Any combinatorial type has a distinguished edge defined by the
component containing the zeroth marking $z_0$.  Thus the combinatorial
type of a nodal disk with markings is a rooted tree.  Let $M_{n,{T}}$
denote the set of isomorphism classes of stable nodal marked disks of
combinatorial type ${T}$, and
$ \ol{M}_n = \bigcup_{{T}} M_{n,{T}} .$
$\ol{M}_{n}$ can be identified with a part of the real locus of the
Grothendieck-Knudsen moduli space $\ol{{M}}_{n+1}(\C)$ of stable genus
zero marked complex curves.  The topology on $\ol{M}_{n+1}(\C)$ has an
explicit description in terms of cross-ratios \cite[Appendix
  D]{mcd-sal}, hence so does the topology on $\ol{M}_n$.  The
cross-ratio of four distinct points $w_1,w_2,w_3,w_4 \in \C$ is
$$ \rho_4(w_1,w_2,w_3,w_4) = \frac{(w_2 - w_3)(w_4 - w_1)}{(w_1 -
w_2)(w_3 - w_4)} $$
and represents the image of $w_4$ under the fractional linear transformation that sends $w_1$ to 0, $w_2$ to 1, and $w_3$ to $\infty$. $\rho_4$ is invariant under the action of $SL(2,\C)$ on $\C$ by
fractional linear transformations.  By identifying $\P^1(\C) \to \C \cup
\{ \infty \}$ and using invariance we obtain an extension of $\rho_4$ to
$\P^1(\C)$, that is, a map
$$ \rho_4: \{ (w_1,w_2,w_3,w_4) \in (\P^1(\C))^4, \ \ i \neq j \implies w_i
\neq w_j \} \to \C - \{ 0 \} .$$
$\rho_4$ naturally extends to the geometric invariant theory quotient
$$ (\P^1(\C))^4 \qu SL(2,\C) = \{ (w_1,w_2,w_3,w_4), \ \text{no more than
two points equal} \} / SL(2,\C) $$
by setting
\begin{equation} \label{extend}
 \rho_4(w_1,w_2,w_3,w_4) = \left\{ \begin{array}{ll} 0 & \text{\ if \ } w_2
= w_3 \text{\ or \ } w_1 = w_4 \\ 
1 & \text{\ if \ } w_1 = w_3
\text{\ or \ } w_2= w_4\\
\infty & \text{\ if \ } w_1 = w_2
\text{\ or \ } w_3= w_4 \end{array} \right\} \end{equation}
and defines an isomorphism from $(\P^1(\C))^4 \qu SL(2,\C)$ to $\P^1(\C) .$
Let $\R^4_+ \subset \R^4$ denote the subset of distinct points
$(w_1,w_2,w_3,w_4) \in \R^4$ in cyclic order.  The restriction of
$\rho_4$ to $\R^4_+$ takes values in $(-\infty,0)$ and is invariant
under the action of $SL(2,\R)$ by fractional linear transformations.
Hence it descends to a map
$ (\R^4)_+ / SL(2,\R) \to (-\infty,0) .$
Let $D$ denote the unit disk, and identify $D \setminus \{-1\} $ with the half
plane $\H$ by $ z \mapsto 1/(z+1)$.  Using invariance one constructs an
extension
$ \rho_4: (\partial D)_+^4 / SL(2,\R) = M_4 \to (-\infty,0) $
where $(\partial D)_+^4$ is the set of distinct points on $\partial D$
in counterclockwise cyclic order.  $\rho_4$ admits an extension to $\ol{M}_4$ 
via \eqref{extend} and so defines an isomorphism 
$ \rho_4: \ol{M}_4 \to [-\infty,0] .$
For any distinct indices $i,j,k,l$ the cross-ratio $\rho_{ijkl}$ is
the function
$$ \rho_{ijkl}: M_n \to \R, \ \ \ [w_0,\ldots,w_n] \mapsto
\rho_4(w_i,w_j,w_k,w_l) .$$
Extend $\rho_{ijkl}$ to $\ol{M}_n$ as follows.  Let ${T}(ijkl) \subset
{T}$ be the subtree whose ending edges are the semi-infinite edges
$i,j,k,l$.  The subtree ${T}(ijkl)$ is one of the three types in
Figure \ref{edge}.
\begin{figure}[h]
\setlength{\unitlength}{0.00027489in}
\begingroup\makeatletter\ifx\SetFigFont\undefined%
\gdef\SetFigFont#1#2#3#4#5{%
  \reset@font\fontsize{#1}{#2pt}%
  \fontfamily{#3}\fontseries{#4}\fontshape{#5}%
  \selectfont}%
\fi\endgroup%
{\renewcommand{\dashlinestretch}{30}
\begin{picture}(11360,3975)(0,-10)
\path(9499.066,3746.360)(9393.000,3810.000)(9456.640,3703.934)
\path(9393,3810)(10293,2910)(11193,3810)
\path(11129.360,3703.934)(11193.000,3810.000)(11086.934,3746.360)
\path(10293,2910)(10293,1110)(9393,210)
\path(9456.640,316.066)(9393.000,210.000)(9499.066,273.640)
\path(10293,1110)(11193,210)
\path(11086.934,273.640)(11193.000,210.000)(11129.360,316.066)
\put(9318,3795){\makebox(0,0)[lb]{{{$i$}}}}
\put(11268,3755){\makebox(0,0)[lb]{{{$l$}}}}
\put(11238,15){\makebox(0,0)[lb]{{{$k$}}}}
\put(9303,30){\makebox(0,0)[lb]{{{$j$}}}}
\path(241.066,2817.360)(135.000,2881.000)(198.640,2774.934)
\path(135,2881)(1035,1981)(135,1081)
\path(198.640,1187.066)(135.000,1081.000)(241.066,1144.640)
\path(1035,1981)(2835,1981)(3735,2881)
\path(3671.360,2774.934)(3735.000,2881.000)(3628.934,2817.360)
\path(2835,1981)(3735,1081)
\path(3628.934,1144.640)(3735.000,1081.000)(3671.360,1187.066)
\put(45,3016){\makebox(0,0)[lb]{{{$i$}}}}
\put(3615,3001){\makebox(0,0)[lb]{{{$l$}}}}
\put(3675,841){\makebox(0,0)[lb]{{{$k$}}}}
\put(15,871){\makebox(0,0)[lb]{{{$j$}}}}
\path(5361.066,3322.360)(5255.000,3386.000)(5318.640,3279.934)
\path(5255,3386)(6605,2036)(5255,686)
\path(5318.640,792.066)(5255.000,686.000)(5361.066,749.640)
\path(7891.360,3279.934)(7955.000,3386.000)(7848.934,3322.360)
\path(7955,3386)(6605,2036)(7955,686)
\path(7848.934,749.640)(7955.000,686.000)(7891.360,792.066)
\put(5090,3416){\makebox(0,0)[lb]{{{$i$}}}}
\put(7955,3466){\makebox(0,0)[lb]{{{$l$}}}}
\put(7985,476){\makebox(0,0)[lb]{{{$k$}}}}
\put(5105,536){\makebox(0,0)[lb]{{{$j$}}}}
\end{picture}
}
\caption{Cross-ratios by combinatorial type: for the first type, $\rho_{ijkl}(S) = -\infty$, for the second type $\rho_{ijkl}(S) \in (-\infty, 0)$, and for the third type $\rho_{ijkl}(S) = 0$.}\label{edge}
\end{figure}
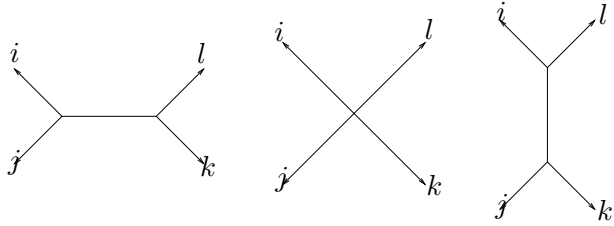
In the first resp. third case, we define $ \rho_{ijkl}(S) = -\infty
\ \ \ \text{resp.} \ 0.$ In the second case, let
$\ol{w}_i,\ol{w}_j,\ol{w}_k,\ol{w}_l$ be the points on the component
where the four branches meet and define $ \rho_{ijkl}(S) = \rho(
\ol{w}_i,\ol{w}_j,\ol{w}_k,\ol{w}_l) .$  
%
\noindent Properties of $\rho_{ijkl}$ that follow from elementary facts about 
cross-ratios \cite[Appendix D]{mcd-sal} are

\begin{proposition}  
\begin{enumerate} 
\item {\bf (Invariance):} For all marked nodal disks $S$, and for all
  $\phi \in SL(2,\R)$, $\rho_{ijkl}(\phi(S)) = \rho_{ijkl}(S)$.
\item {\bf (Symmetry):} $\rho_{jikl} = \rho_{ijlk} = 1 - \rho_{ijkl}$, and $\rho_{ikjl} = \rho_{ijkl}( \rho_{ijkl}-1)$.
\item {{\bf (Normalization):}} $\rho_{ijkl} = \left\{
\begin{array}{cc} 
\infty, & \mbox{if}\ i=j\ \mbox{or}\ k=l,\\ 1, &
\mbox{if}\ i=k\ \mbox{or}\ j=l,\\ 0, & \mbox{if}\ i=l\ \mbox{or}\ j=k.
\end{array}\right.$
\item {\bf (Recursion):} As long as the set $\{1, \infty, \rho_{ijkl}, \rho_{ijkm}\}$ contains three distinct numbers, then 
$\rho_{jklm} = \frac{\rho_{ijkm}-1}{\rho_{ijkm} - \rho_{ijkl}}$
for any five pairwise distinct integers $i, j, k, l, m \in \{0, 1, \ldots, d\}$.  
\end{enumerate}
\end{proposition} 

The collection of functions
$\rho_{ijkl}, \ i < j < k < l$ defines a map of sets
\begin{equation}\label{embedding_md}
 \rho_n: \ol{M}_n \mapsto [-\infty,0]^N, \quad \ N ={d+1\choose 4},
 \end{equation} 
 which is the restriction of the corresponding map $\rho_n: \ol{M}_{n+1}(\C) \to (\P^1(\C))^N$ defined by cross-ratios.
\begin{theorem}[Theorem D.4.5, \cite{mcd-sal}]\label{injection_n}
The map $\rho_n: \ol{M}_{n+1}(\C) \to (\P^1(\C))^N$ is injective, and its image is closed. 
\end{theorem}
\begin{corollary}
The map $\rho_n: \ol{M}_n \mapsto [-\infty,0]^N$ is an embedding, and its image is closed.  
\end{corollary} 
The topology on $\ol{M}_n$ is defined by pulling back the topology on $[-\infty,0]^N$.  With respect to this topology, $\ol{M}_n$ is compact and Hausdorff.  
%
%
Explicit coordinate charts which give $\ol{M}_n$ the structure of a
$(n-2)$ dimensional manifold-with-corners can be defined with
cross-ratios.  There is a canonical partial order on the combinatorial
types, and we write $T_0\leq T_1$ to mean that $T_1$ is obtained from
$T_0$ by contracting a subset of finite edges of $T_0$, in other
words, there is a surjective morphism of trees from $T_0$ to $T_1$.  Let
%
\[
\ol{M}_{n,\leq T_1} := \underset{T_0\leq T_1}{\cup} M_{n,T_0} \subset
\ol{M}_n.
\] 

\begin{definition}  \label{crossratios} A {\em cross-ratio chart} for
a combinatorial type ${T}$ is a map 
$$\psi_T: \ol{M}_{n,\leq T} \to
(0,\infty)^{n-2-|E|}\times [0,\infty)^{|E|}$$ 
where $|E|$ is the number of interior edges of ${T}$, given by
\begin{enumerate}
\item $n-2-|E|$ coordinates taking values in $(0,\infty)$, obtained by choosing $m- 3$ coordinates $-\rho_{ijkl}$ for each disk component with $m$ marked or singular points,
\item $|E|$ coordinates with values in $[0,\infty)$, obtained by choosing a coordinate $-\rho_{ijkl} = 0$ for each internal edge such that a $\rho_{ijkl}=0$ for any combinatorial type modeled on that edge.  
\end{enumerate}
\end{definition} 

\begin{theorem}[Theorem D.5.1 in \cite{mcd-sal}]\label{chartdisks}
For any combinatorial type $T$, suppose that $n-2$ cross-ratios have
been chosen as prescribed by (a), (b) above. Then, in the open set
$\ol{M}_{n,\leq {T}}$, all cross-ratios are smooth functions of those
chosen. Hence $\ol{M}_n$ is a smooth manifold-with-corners of real
dimension $n-2$.
\end{theorem}


\label{metricribbon}

The associahedra have another geometric realization as {\em metric
  trees}, introduced in Boardman-Vogt \cite{boardman-vogt}.  Here we
follow the presentation in \cite{zero-loop}.
\begin{definition}
A {\em rooted metric ribbon tree} consists of
\begin{enumerate}
\item a finite tree $T = (V(T),\ol{E}(T))$ where $\ol{E}(T)$ is the
  union of a set $E(T)$ of {\em finite edges} incident to two vertices
  and a set $E_\infty(T) = \{ e_0,\ldots, e_n \}$ of {\em
    semi-infinite edges}, each of which is incident to a single
  vertex;
\item a cyclic ordering on the edges $\{ e \in \ol{E}(T), v \in e \}$ at
  each vertex $v \in V(T)$;
\item a distinguished edge $e_0 \in E_\infty(T)$, called the {\em
  root}; the other semi-infinite edges are called {\em leaves}.
\item a {\em metric} $\lambda: E(T) \to (0,\infty)$ 
\end{enumerate}
A tree is {\em stable} if each vertex has valence at least $3$.
\end{definition}

\noindent Given a rooted ribbon tree $T$ we denote by $\WW_{n,T}$ the
set of all metrics $\lambda: E(T) \to \R_+$.  The space of
all stable rooted metric ribbon trees with $n$ leaves is denoted
$$ \WW_n = \bigcup\limits_T \WW_{n,T}. $$
There is a natural topology on $\WW_n$, which allows the collapse of
edges whose lengths approach zero in a sequence.  The closure of
$\WW_{n,T}$ in $\WW_n$ is given by
\[
\ol{\WW}_T = \bigcup\limits_{T^\prime \leq T} \WW_{n,T^\prime}.
\]
\noindent 
Each cell $\WW_{n,T}$ is compactified by allowing the edge lengths to
be infinite. We denote the induced compactification of $\WW_n$ by
$\ol{\WW}_n$.
 The following theorem is
well-known:
\begin{theorem}\label{homeo2}
There exists a homeomorphism $\Theta: \ol{\WW}_n \to \ol{M}_{n}$ such that for any combinatorial type ${T}$, $\Theta(\ol{\WW}_{n,T})$ intersects $M_{n,{T}}$ in a single point.
\end{theorem}
\noindent In other words, the realization as metric trees is dual, in
a CW-sense, to the realization as marked disks.  We give a proof of
the corresponding statement for the multiplihedra in the next section.

\section{The multiplihedra} 

\label{multiplihedron}

Stasheff \cite{stasheffbook} introduced a family of $CW$-complexes
called the {\em multiplihedra}, which play the same role for maps of loop spaces as the associahedra play in the recognition principle
for loop spaces.  The $n$-th multiplihedron $J_n$ is a complex of
dimension $n - 1$ whose vertices correspond to ways of bracketing $n$
variables $x_1,\ldots,x_n$ and applying an operation, say $f$.  The
multiplihedron $J_3$ is the hexagon shown in Figure \ref{K31}.
\begin{figure}[h]
\includegraphics[height=2in]{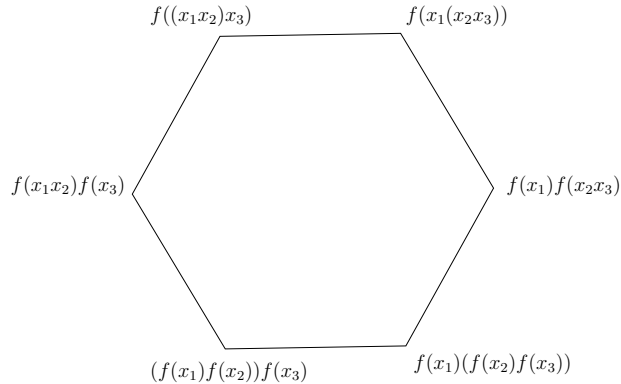}
\caption{ Vertices of $J_{3}$}
\label{K31}
\end{figure}

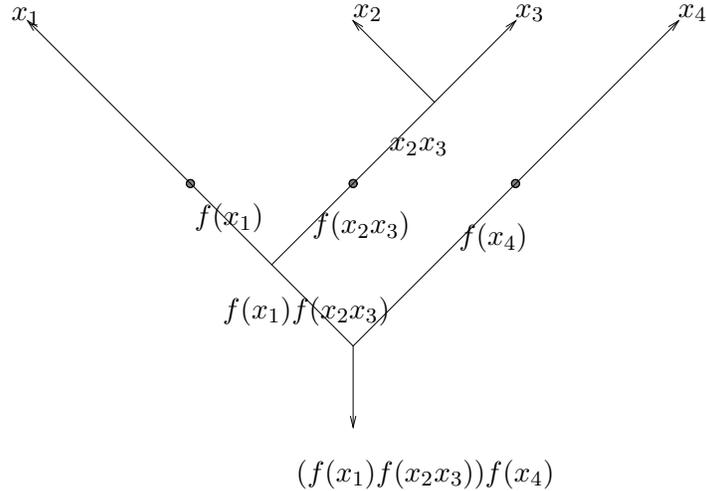
\begin{figure}[ht]
\setlength{\unitlength}{0.00047489in}
\begingroup\makeatletter\ifx\SetFigFont\undefined%
\gdef\SetFigFont#1#2#3#4#5{%
  \reset@font\fontsize{#1}{#2pt}%
  \fontfamily{#3}\fontseries{#4}\fontshape{#5}%
  \selectfont}%
\fi\endgroup%
{\renewcommand{\dashlinestretch}{30}
\begin{picture}(7625,5428)(0,-10)
\texture{44555555 55aaaaaa aa555555 55aaaaaa aa555555 55aaaaaa aa555555 55aaaaaa 
	aa555555 55aaaaaa aa555555 55aaaaaa aa555555 55aaaaaa aa555555 55aaaaaa 
	aa555555 55aaaaaa aa555555 55aaaaaa aa555555 55aaaaaa aa555555 55aaaaaa 
	aa555555 55aaaaaa aa555555 55aaaaaa aa555555 55aaaaaa aa555555 55aaaaaa }
\put(1995,3448){\shade\ellipse{90}{90}}
\put(1995,3448){\ellipse{90}{90}}
\put(5595,3448){\shade\ellipse{90}{90}}
\put(5595,3448){\ellipse{90}{90}}
\put(3795,3448){\shade\ellipse{90}{90}}
\put(3795,3448){\ellipse{90}{90}}
\path(301.066,5184.360)(195.000,5248.000)(258.640,5141.934)
\path(195,5248)(2895,2548)(4695,4348)(3795,5248)
\path(3901.066,5184.360)(3795.000,5248.000)(3858.640,5141.934)
\path(5531.360,5141.934)(5595.000,5248.000)(5488.934,5184.360)
\path(5595,5248)(4695,4348)
\path(2895,2548)(3795,1648)(7395,5248)
\path(7331.360,5141.934)(7395.000,5248.000)(7288.934,5184.360)
\path(3795,1648)(3795,748)
\path(3765.000,868.000)(3795.000,748.000)(3825.000,868.000)
\put(3795,5248){\makebox(0,0)[lb]{{{$x_2$}}}}
\put(5595,5248){\makebox(0,0)[lb]{{{$x_3$}}}}
\put(7395,5248){\makebox(0,0)[lb]{{{$x_4$}}}}
\put(15,5203){\makebox(0,0)[lb]{{{$x_1$}}}}
\put(3165,73){\makebox(0,0)[lb]{{{($f(x_1)f(x_2x_3))f(x_4)$}}}}
\put(4200,3763){\makebox(0,0)[lb]{{{$x_2x_3$}}}}
\put(3345,2818){\makebox(0,0)[lb]{{{$f(x_2x_3)$}}}}
\put(2040,2908){\makebox(0,0)[lb]{{{$f(x_1)$}}}}
\put(4965,2683){\makebox(0,0)[lb]{{{$f(x_4)$}}}}
\put(2355,1873){\makebox(0,0)[lb]{{{$f(x_1)f(x_2x_3)$}}}}
\end{picture}
}
\caption{Tree for $(f(x_1)f(x_2 x_3)) f(x_4)$}
\label{treeex} \end{figure}

The facets of $J_n$ are of two types.  First, there are the images
of the inclusions
$$ J_{i_1} \times \ldots \times J_{i_j} \times K_j \to J_{n} $$
for partitions $i_1 + \ldots + i_j = n$, and secondly the images of
the inclusions
$$ J_{n-e+1} \times K_e \to J_n $$
for $2 \leq e \leq n$.  One constructs the multiplihedron inductively
starting from setting $J_{2}$ and $K_{3}$ equal to closed intervals.

Each vertex corresponds to a rooted tree with two types of vertices,
the first a trivalent vertex corresponding to a bracketing of two
variables and the second a bivalent vertex corresponding to an
application of $f$, see Figure \ref{treeex}.  Dualizing the rooted
tree gives a triangulation of the $n + 1$-gon together with a
partition of the two-cells into two types, depending on whether they
occur before or after a bivalent vertex in a path from the root, see
Figure \ref{parex}.
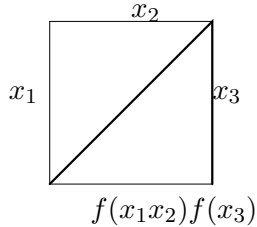
\begin{figure}[h]
\setlength{\unitlength}{0.00047489in}
\begingroup\makeatletter\ifx\SetFigFont\undefined%
\gdef\SetFigFont#1#2#3#4#5{%
  \reset@font\fontsize{#1}{#2pt}%
  \fontfamily{#3}\fontseries{#4}\fontshape{#5}%
  \selectfont}%
\fi\endgroup%
{\renewcommand{\dashlinestretch}{30}
\begin{picture}(2495,2503)(0,-10)
\path(465,2323)(2265,2323)(2265,523)
	(465,523)(465,2323)
\thicklines
\path(465,523)(2265,2323)(2265,523)
\put(15,1423){\makebox(0,0)[lb]{{{$x_1$}}}}
\put(1365,2323){\makebox(0,0)[lb]{{{$x_2$}}}}
\put(2265,1423){\makebox(0,0)[lb]{{{$x_3$}}}}
\put(915,73){\makebox(0,0)[lb]{{{$f(x_1x_2)f(x_3)$}}}}
\end{picture}
}
\caption{Triangulation corresponding to $f(x_1x_2)f(x_3)$}
\label{parex}
\end{figure}
The edges of $J_{n}$ are of two types: 
(a) A change in bracketing
$ \ldots x_{i-1}(x_i x_{i+1}) \ldots  \mapsto (x_{i-1} x_i) x_{i+1}  $
or vice-versa;
(b) A move of the form
$ \ldots f(x_i x_{i+1}) \ldots \mapsto f(x_i) f(x_{i+1}) \ldots $
or vice versa, which corresponds to moving one of the 
bivalent vertices past a trivalent vertex, after
which it becomes a pair of bivalent vertices, or vice-versa;
see Figure \ref{splitting}. 
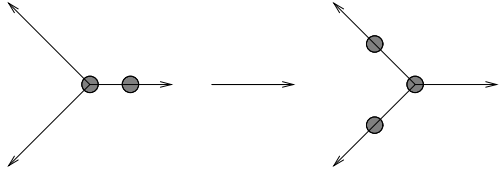
\begin{figure}[h]
\setlength{\unitlength}{0.00047489in}
\begingroup\makeatletter\ifx\SetFigFont\undefined%
\gdef\SetFigFont#1#2#3#4#5{%
  \reset@font\fontsize{#1}{#2pt}%
  \fontfamily{#3}\fontseries{#4}\fontshape{#5}%
  \selectfont}%
\fi\endgroup%
{\renewcommand{\dashlinestretch}{30}
\begin{picture}(5424,1839)(0,-10)
\texture{44555555 55aaaaaa aa555555 55aaaaaa aa555555 55aaaaaa aa555555 55aaaaaa 
	aa555555 55aaaaaa aa555555 55aaaaaa aa555555 55aaaaaa aa555555 55aaaaaa 
	aa555555 55aaaaaa aa555555 55aaaaaa aa555555 55aaaaaa aa555555 55aaaaaa 
	aa555555 55aaaaaa aa555555 55aaaaaa aa555555 55aaaaaa aa555555 55aaaaaa }
\put(1362,912){\shade\ellipse{180}{180}}
\put(1362,912){\ellipse{180}{180}}
\put(4062,1362){\shade\ellipse{180}{180}}
\put(4062,1362){\ellipse{180}{180}}
\put(4062,462){\shade\ellipse{180}{180}}
\put(4062,462){\ellipse{180}{180}}
\put(912,912){\shade\ellipse{180}{180}}
\put(912,912){\ellipse{180}{180}}
\put(4512,912){\shade\ellipse{180}{180}}
\put(4512,912){\ellipse{180}{180}}
\path(2262,912)(3162,912)
\path(2262,912)(3162,912)
\path(3042.000,882.000)(3162.000,912.000)(3042.000,942.000)
\path(118.066,1748.360)(12.000,1812.000)(75.640,1705.934)
\path(12,1812)(912,912)(12,12)
\path(75.640,118.066)(12.000,12.000)(118.066,75.640)
\path(912,912)(1812,912)
\path(1692.000,882.000)(1812.000,912.000)(1692.000,942.000)
\path(4512,912)(5412,912)
\path(5292.000,882.000)(5412.000,912.000)(5292.000,942.000)
\path(3718.066,1748.360)(3612.000,1812.000)(3675.640,1705.934)
\path(3612,1812)(4512,912)(3612,12)
\path(3675.640,118.066)(3612.000,12.000)(3718.066,75.640)
\end{picture}
}

\caption{Splitting of bivalent vertices}
\label{splitting} \end{figure}

\section{Quilted disks}

\begin{definition} 
A {\em quilted disk} is a 
closed disk $D \subset \C$ together with a circle $C \subset D$ (the
{\em seam} of the quilt) tangent to a unique point in the boundary.
Thus $C$ divides the interior of $D$ into two components.  Given
quilted disks $(D_0,C_0)$ and $(D_1,C_1)$, a {\em isomorphism} from
$(D_0,C_0)$ to $(D_1,C_1)$ is a holomorphic isomorphism $D_0 \to D_1$
mapping $C_0$ to $C_1$.  Any quilted disk is isomorphic to the pair
$(D,C)$ where $D$ is the unit disk in the complex plane and $C$ the
circle of radius $1/2$ passing through $1$ and $0$.  Thus the
automorphism group of $(D,C)$ is canonically isomorphic to the group
$T \subset SL(2,R)$ of translations by real numbers.

Let $n \ge 2$ be an integer.  A {\em quilted disk with $n+1$ markings
  on the boundary} consists of a disk $D \subset \C$ (which we may
take to be the unit disk), distinct points $z_0,\ldots,z_n \in
\partial D$ and a circle $C \subset D$ tangent to $z_0$, of radius
between 0 and 1.  A {\em morphism of quilted disks} from
$(D_0,C_0;z_0,\ldots,z_n) \to (D_1,C_0;w_0,\ldots,w_n)$ is a
holomorphic isomorphism $D_0 \to D_1$ mapping $C_0$ to $C_1$ and $z_j$
to $w_j$ for $j = 0,\ldots,n$.  
\end{definition} 

Let $M_{n,1}$ be the set of isomorphism classes of $n+1$-marked
quilted disks.  We compactify $M_{n,1}$ by allowing nodal quilted
disks whose combinatorial type is described as follows.  

\begin{definition}
A {\em colored, rooted ribbon tree} is a ribbon tree $T=(\ol{E}(T),
V(T))$ together with a distinguished subset $V_{col}(T) \subset V(T)$
of {\em colored vertices}, such that in any non-self-crossing path
from a leaf $e_i$ to the root $e_0$, exactly one vertex is a colored
vertex.
\end{definition}

\begin{definition} 
A {\em nodal $(d+1)$-quilted disk} $S$ is a collection of quilted and
unquilted marked disks, identified at pairs of points on the boundary.
The combinatorial type of $S$ is a colored rooted ribbon tree ${T}$, 
where the colored vertices represent quilted disks, and the remaining
vertices represent unquilted disks.
A nodal quilted disk is {\em stable} if and only if
\begin{enumerate}
\item Each quilted disk component contains at least $2$ singular or
marked points;
\item Each unquilted disk component contains at least
$3$ singular or marked points.
\end{enumerate}
\end{definition} 
\noindent Thus the automorphism group of any disk component of a stable disk is
trivial, and from this one may derive that the automorphism group of
any stable $n+1$-marked nodal quilted disk is also trivial.  

The appearance of the two kinds of disks can be explained in the
language of bubbling as in \cite{mcd-sal}.  Suppose that $S_\alpha$ is
a sequence of quilted disks.  We identify the complement of $z_0$ with
the upper half-space $\H$, so that the circle $C_\alpha$ becomes a horizontal
line $L_\alpha \subset \H$.  After a sequence of automorphisms
$\varphi_\alpha$, we may assume that $z_{1,\alpha} - z_{n,\alpha}$ is
constant.  If the line $L_\alpha$ approaches the real axis, or two
points $z_{i,\alpha}, z_{j,\alpha}$ converge then we re-scale so that
the distances remain finite and encode the limit of the re-scaled data
as a bubble.  There are three-different types of bubbles: either
$L_\alpha \to \partial \H$ in the limit, in which case we say that the
resulting bubble is unquilted, or $L_\alpha$ approaches a fixed line
$L_\infty$, in which case the bubble is a quilted disk, or $L_\alpha$ goes 
to $\infty$, in which case the bubble is also unquilted.  Thus the
limiting sequence is a bubble tree, whose bubbles are of the types
discussed above.

Let $\ol{M}_{n,1}$ denote the set of isomorphism classes of stable
$n+1$-marked nodal quilted disks.  For example, $\ol{M}_{3,1}$ is a
hexagon.
\begin{figure}[h]
\includegraphics[height=3in]{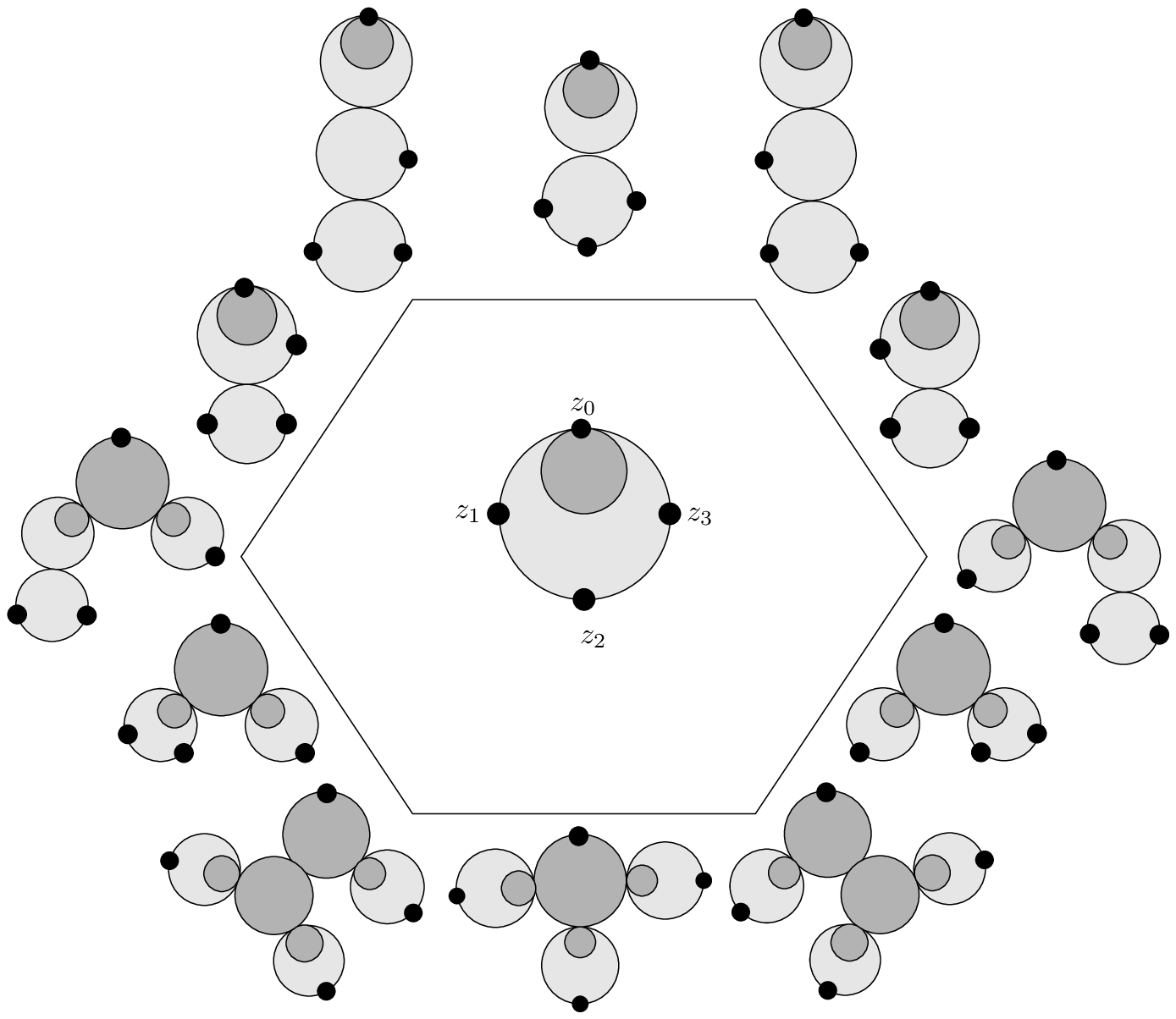}
\caption{$\ol{M}_{3,1}$}
\label{M31}
\end{figure}

\section{The canonical embedding}  

$\ol{M}_{n,1}$ admits a canonical embedding into a product of closed
intervals via a natural generalization of cross-ratios.  Let $D$
denote the unit disk, $C$ a circle in $D$ passing through a unique
point $z_0$ and $z_1,z_2 \in D$ points in $D$ such that $z_0,z_1,z_2$
are distinct.  Let $w$ be a point in $C$ not equal to $z_0$.  Define
$ \rho_{3,1}(D,C,z_1,z_2) = \Im( \rho(z_0,z_1,z_2,w)) ,$
the imaginary part of $ \rho(z_0,z_1,z_2,w)$.
$\rho_{3,1}$ is independent of the choice of $w$ and invariant under
the group of automorphisms of the disk and so defines a map $
\rho_{3,1}: M_{3,1} \to (0,\infty) .$ We extend $\rho_{3,1}$ to
$\ol{M}_{3,1}$ by setting $\rho_{3,1}(S) = 0$ if $S$ is the $3$-marked
quilted nodal disk with three components, and $\rho_{3,1}(S) = \infty$
if $S$ is the $3$-marked nodal disk with two components.  Thus
$\rho_{3,1}$ extends to a bijection
$$ \rho_{3,1}: \ol{M}_{3,1} \to [0,\infty ].$$
More generally, given $n \ge 3$ and a pair $i,j$ of distinct, non-zero
vertices, let ${T}_{ij}$ denote the minimal connected subtree of
${T}$ containing the semi-infinite edges corresponding to
$z_i,z_j,z_0$.  There are three possibilities for ${T}_{ij}$,
depending on whether the quilted vertex appears closer or further away
than the trivalent vertex from $z_0$, or equals the trivalent vertex.
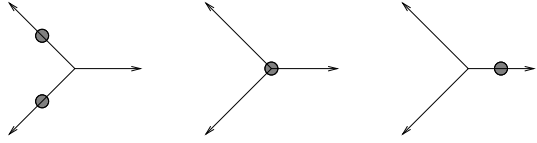
\begin{figure}[h]
\setlength{\unitlength}{0.00037489in}
\begingroup\makeatletter\ifx\SetFigFont\undefined%
\gdef\SetFigFont#1#2#3#4#5{%
  \reset@font\fontsize{#1}{#2pt}%
  \fontfamily{#3}\fontseries{#4}\fontshape{#5}%
  \selectfont}%
\fi\endgroup%
{\renewcommand{\dashlinestretch}{30}
\begin{picture}(7224,1839)(0,-10)
\texture{44555555 55aaaaaa aa555555 55aaaaaa aa555555 55aaaaaa aa555555 55aaaaaa 
	aa555555 55aaaaaa aa555555 55aaaaaa aa555555 55aaaaaa aa555555 55aaaaaa 
	aa555555 55aaaaaa aa555555 55aaaaaa aa555555 55aaaaaa aa555555 55aaaaaa 
	aa555555 55aaaaaa aa555555 55aaaaaa aa555555 55aaaaaa aa555555 55aaaaaa }
\put(462,1362){\shade\ellipse{180}{180}}
\put(462,1362){\ellipse{180}{180}}
\put(462,462){\shade\ellipse{180}{180}}
\put(462,462){\ellipse{180}{180}}
\put(3612,912){\shade\ellipse{180}{180}}
\put(3612,912){\ellipse{180}{180}}
\put(6762,912){\shade\ellipse{180}{180}}
\put(6762,912){\ellipse{180}{180}}
\path(2818.066,1748.360)(2712.000,1812.000)(2775.640,1705.934)
\path(2712,1812)(3612,912)(2712,12)
\path(2775.640,118.066)(2712.000,12.000)(2818.066,75.640)
\path(3612,912)(4512,912)
\path(4392.000,882.000)(4512.000,912.000)(4392.000,942.000)
\path(5518.066,1748.360)(5412.000,1812.000)(5475.640,1705.934)
\path(5412,1812)(6312,912)(5412,12)
\path(5475.640,118.066)(5412.000,12.000)(5518.066,75.640)
\path(6312,912)(7212,912)
\path(7092.000,882.000)(7212.000,912.000)(7092.000,942.000)
\path(912,912)(1812,912)
\path(1692.000,882.000)(1812.000,912.000)(1692.000,942.000)
\path(118.066,1748.360)(12.000,1812.000)(75.640,1705.934)
\path(12,1812)(912,912)(12,12)
\path(75.640,118.066)(12.000,12.000)(118.066,75.640)
\end{picture}
}
\caption{Tree types for $J_{3}$}
\end{figure}
In the first, resp. third case define $\rho_{ij}(S) = 0$ resp
$\infty$.  In the second case let $(D,C)$ denote the disk component
corresponding to the trivalent vertex, $w_i,w_j \in \partial D$ the
points corresponding to the images in $\partial D$ of the marked points $z_i,z_j$, and
define
$$ \rho_{ij}(S) = \rho_{3,1}(D,C,w_i,w_j) .$$
The $\rho_{ij}$ have properties very similar to the $\rho_{ijkl}$:

\begin{proposition} For all quilted disks $S$, 
\begin{enumerate}
\item {\bf (Invariance):} for all $\phi \in SL(2,\R)$,
  $\rho_{ij}(\phi(S)) = \rho_{ij}(S)$.
\item{\bf (Symmetry):} $\rho_{ij}(S) = -\rho_{ji}(S)$.
\item{{\bf (Normalization):}} 
$\rho_{ij}(S) = \left\{
\begin{array}{cc} 
\infty, & \mbox{if}\ i\neq j\ \mbox{and}\ L = \R + i \infty,\\
0, & \mbox{if}\ i\neq j\ \mbox{and}\ L = \R + i 0 .
\end{array}\right.$\\
\item {\bf (Recursion):} 
$\rho_{ik}(S) = \frac{\rho_{ij}(S)}{\rho_{jk}(S)}$
\item {\bf (Relations):}  $ \rho_{jk} =
  \frac{\rho_{ij}}{-\rho_{ijk0}} \label{rel1}, \quad \rho_{ik} =
  \frac{\rho_{ij}}{1-\rho_{ijk0}}\label{rel2}.$
\end{enumerate}
\end{proposition} 

By the invariance property, $\rho_{ij}$ descends to a map 
$$ \ol{M}_{n,1} \to [0,\infty] .$$
In addition, for any four distinct indices $i,j,k,l$ we have the
cross-ratio
$ \rho_{ijkl}: \ol{M}_{n,1} \to [0,\infty] $
defined in the previous section, obtained by treating the quilted disk
component as an ordinary component.

\begin{theorem}  The map 
$$ \rho_{n,1}: \ol{M}_{n,1} \to [-\infty,0]^N \times
  [0,\infty]^{{n(n-1)/2}}, \quad N= {d+1\choose{4}} $$
obtained from all the cross-ratios is injective, and its image is closed. 
\end{theorem}  

\begin{proof} The proof is similar that of Theorem \ref{injection_n}.
In fact, it is a corollary of Theorem \ref{bijection} in Section
\ref{complex_multi}, which deals with a complex space
$\ol{M}_{n,1}(\C)$ in which $\ol{M}_{n,1}$ sits as a part of the real
locus.
\end{proof}

We define the topology on $\ol{M}_{n,1}$ by pulling back the topology
on the codomain.  Since the codomain is Hausdorff and compact,

\begin{corollary}  
$\ol{M}_{n,1}$ is Hausdorff and compact. 
\end{corollary}

\begin{remark}  The maps $\ol{M}_{n,1} \to \ol{M}_4$,
$\ol{M}_{n,1} \to \ol{M}_{3,1}$ are special cases of forgetful
  morphisms: For any subset $I \subset \{ 0,\ldots,n \}$ of size $k$
  we have a map
$ \ol{M}_{n,1} \mapsto \ol{M}_k $
obtained by forgetting the position of the circle and collapsing all
unstable components.  Similarly, for any subset $J \subset \{1,\ldots,
n\}$ of size $l$ we have a map
$ \ol{M}_{n,1} \mapsto \ol{M}_{l,1} $
obtained by forgetting the positions of $z_j, j \notin J$ and
collapsing all unstable disk components.  The topology on
$\ol{M}_{n,1}$ is the minimal topology such that all forgetful
morphisms are continuous and the topology on $\ol{M}_{3,1} \cong
[0,\infty]$, $\ol{M}_4 \cong [0,\infty]$ is induced by the
cross-ratio.
\end{remark}

The full collection of cross-ratios contains a large amount of
redundant information.  In the remainder of this section we discuss
certain ``minimal sets'' of cross-ratios, to be used later.  Let ${T}$
be a combinatorial type in $\ol{M}_{n,1}$.

\begin{definition}  A {\em cross-ratio chart}
associated to ${T}$ is a map $\psi_T: \ol{M}_{n,1,\leq T} \to
(0,\infty)^p\times [0,\infty)^q$ for some $p,q \geq 0$ given by
\begin{enumerate}
\item $p$ coordinates taking values in $(0,\infty)$, obtained by taking $m-3$ coordinates of the form $-\rho_{abcd}$ or $\rho_{ab}$  for each disk component that has $m$ special
  features, where a special feature is either a marked point, a nodal
  point, or an inner circle of radius $ 0 < r < 1$;
\item $q$ coordinates taking values in $[0,\infty)$, obtained by
  choosing (i) a coordinate $-\rho_{abcd} $ for each finite edge in
  $T$, such that a combinatorial type has that edge if and only if
  $\rho_{abcd}=0$; (ii) a coordinate $\rho_{ab} $ for each finite edge
  in ${T}$ that is incident to a \emph{bivalent} colored vertex from
  above, such that $\rho_{ab}=0$ for every combinatorial type modeled
  on that edge; (iii) a coordinate $1/\rho_{ab}$ for each finite edge
  in ${T}$ that is incident to a bivalent colored vertex from below,
  such that $1/\rho_{ab} = 0$ for every combinatorial type modeled on
  that edge.
\end{enumerate}%
\end{definition} 

\begin{proposition}  Let $\psi_T$ be as above. 
On $\ol{M}_{n,1,\leq T}$ all cross-ratios $\rho_{ijkl}$ and
$\rho_{ij}$ are compositions of smooth functions with $\psi_T$.
\end{proposition}

\begin{proof}
First we prove that all cross-ratios of the form $\rho_{ijkl}$ are
smooth functions of those in the chart associated to ${T}$.  Let $T'$
be the combinatorial type in $\ol{M}_n$ obtained by forgetting colored
vertices.  Taking all cross-ratios of the form $\rho_{ijkl}$ in the
chart associated to ${T}$ is almost a chart for $T'$ in the sense of
Definition \ref{crossratios}, the only chart coordinates that might be
missing correspond to edges whose pre-image in $T$ had a bivalent
colored vertex.  For each bivalent vertex in $T$, we can assume that
the lower edge has coordinate $\rho_{i, j} = \infty$ and the upper
edge is either $\rho_{jk} = 0$ or $\rho_{hi} = 0$.  Assuming the first
case, relation (\ref{rel1}) holds and $\rho_{i j k0} = -\rho_{ij} /
\rho_{jk}$, which expresses $\rho_{ij k 0}$ as a smooth function of
the chart coordinates, and $\rho_{ijk0}$ is a valid chart coordinate
for the edge in $T'$. The other case is very similar, by relation
(\ref{rel1}), $\rho_{h i j 0} = - \rho_{hi} / \rho_{ij}$, which
expresses $\rho_{hij0}$ as a smooth function of the chart coordinates,
and $\rho_{hij0}$ is a chart coordinate for the corresponding edge in
$T'$. Thus we get a chart for $T'$, so by Theorem \ref{chartdisks} all
cross-ratios of the form $\rho_{abcd}$ are smooth functions of these
coordinates.  Finally, all cross-ratios $\rho_{ab}$ are smooth
functions of the cross-ratios $\rho_{ij}$ in the chart and the
appropriate $\rho_{ijk0}$'s, by {\bf (Relations)}.
\end{proof}

\section{Local structure}  \label{simple_charts_2}

In general, $\ol{M}_{n,1}$ is not CW-isomorphic to a manifold with
corners, but rather has more complicated singularities that we now
describe.  Quilted disks in the interior $M_{n,1}$ can be identified with configurations of
$n$ distinct points $-\infty < z_1 < z_2 < \ldots < z_n < \infty$ in
$\R \subset \C$, together with a horizontal line $L$ in $\H$.
Isomorphisms are transformations of the form $z \mapsto az + b$ for
$a, b \in \R$ such that $a > 0$, i.e. dilation and translation. For
such configurations define coordinates
$(x_1, x_2, \ldots, x_n, y)$ by $x_i = z_{i+1} - z_i$, and $y =
\dist(L, \R)$. A transformation $z \mapsto az + b$ for $a, b \in \R$
sends
$(x_1 , x_2 , \ldots , x_{n-1} , y) \mapsto (a x_1 , a x_2 , \ldots ,
a x_{n-1} , a y),$
so $(x_1 : x_2 : \ldots : x_{n-1} : y)$ are projective coordinates on
${M}_{n,1}$.

Let $T$ be a maximal colored rooted ribbon tree, hence its colored vertices are
bivalent, and all other vertices trivalent.  We construct a
{\em simple ratio chart}
\begin{equation}
 \label{simpleratios}
\phi_T: \ol{M}_{n,1,\leq T} \to \Hom(\Edge(T),\R_{\ge 0 }), \quad [S]
\mapsto (\phi_{T,e}(S) )_{e\in \Edge(T)} \end{equation} 
as follows.  Let $[S] \in \ol{M}_{n,1,\leq T}$.  For each $1\leq i \leq
n-1$, there is a unique vertex of $T$ at which the path from the leaf
$i$ and the leaf $i+1$ back to the root intersect; we label this
vertex $v_i$.  Every trivalent vertex of $T$ can be labeled this
way, so all remaining vertices are colored. The interior
edges of $T$ are of two types: edges that connect two vertices $v_i$
and $v_j$, and edges that connect a vertex $v_i$ to a colored vertex.
Suppose that $e \in \Edge(T)$ connects the vertex $v_i$ with the
vertex $v_j$, with $v_j$ closer to the root (i.e., $v_j$ is above
$v_i$). The vertex $v_i$ labels the unique component of the nodal disk
$S$ on which the markings corresponding to the leaves $z_j$, $z_{j+1}$
and $z_0$ are distinct.  On this component, choose a parametrization
that sends $z_0$ to $\infty$, then label the edge between $v_i$ and
$v_j$ with $\phi_{T,e} = (z_{i+1}-z_i)/(z_{j+1} - z_{j}) = x_i/x_j$.
Note that the label $\phi_{T,e}$ is therefore independent of the choice
of parametrization. If the edge $e$ connects the vertex $v_i$ with a
colored vertex immediately above it, choose the unique component at
which $z_{i+1}$ is distinct from $z_i$, and label $e$ with $\phi_{T,e} =
x_i/y$; if the edge $e$ connects the vertex $v_i$ with a colored
vertex immediately below it, label $e$ with the value $\phi_{T,e} =
y/x_i$.
\begin{figure}[ht]
\includegraphics[height=4in]{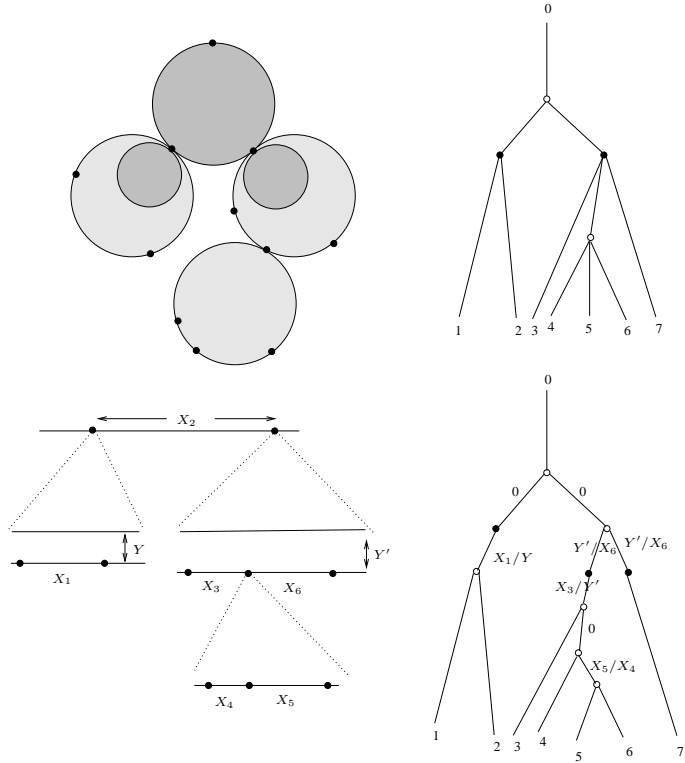}
\caption{Identifying a nodal quilted disk in $\ol{M}_{7,1,\leq T}$ with a balanced labelling of $T$, using simple ratios.}
\label{phi_T}
\end{figure}
We claim that the map $\phi_T$ is defined for all $[S] \in
\ol{M}_{n,1,\leq T}$ with all the ratios $\phi_{T,e}(S)$ landing in
$[0,\infty)$, and with the property that $\phi_{T,e}(S) = 0$ if and only
  if the combinatorial type of $S$ has the edge $e$. To see this,
  recall that if $[S] \in \ol{M}_{n,1,\leq T}$, its combinatorial type
  $T_S$ must be obtained from $T$ by contracting a subset of edges
  (that is, $T_S \leq T$).  In particular, every edge in $T_S$
  corresponds to a unique edge in $T$. Hence, if $e \in \Edge(T)$ connects
  a vertex $v_i$ with a vertex $v_j$ above it, then in $T_S$ either
  $v_i = v_j$ if $e$ is contracted, or the edge $e$ remains.  If $v_i
  = v_j$ in $T_S$, it implies that the disk component of $S$ on which
  $z_j \neq z_{j+1}$ is also the disk component on which $z_i \neq
  z_{i+1}$, hence $x_i > 0$ and $x_j > 0$ and the ratio $\phi_{T,e}(S) =
  x_i/x_j > 0$. If the edge $e$ remains in $T_S$, it means that with
  respect to the markings on the disk component where $z_j \neq
  z_{j+1}$, we have $z_i =z_{i+1}$ and so $\phi_{T,e}(S) = 0/x_j = 0.$

Now if $e \in \Edge(T)$ connects vertex $v_i$ with a colored vertex above
it, then in $T_S$ either $v_i$ becomes a colored vertex, or $e$ is an
edge. In the first case, it implies that the unique component where
$z_i \neq z_{i+1}$ is a quilted component, so parametrizing the
component such that $z_0 = \infty$ the inner circle is a line of
height $y >0$, and $\phi_{T,e}(S) = x_i/y >0$. In the second case, the
unique component where $z_i \neq z_{i+1}$ is unquilted and corresponds
to having the line at height $y = \infty$, so $\phi_{T,e}(S) = x_i/\infty
= 0$. The case of a colored vertex below $v_i$ is similar.
This completes the construction of $\phi_T$. 

\begin{figure}[h]
\includegraphics[height=1.5in]{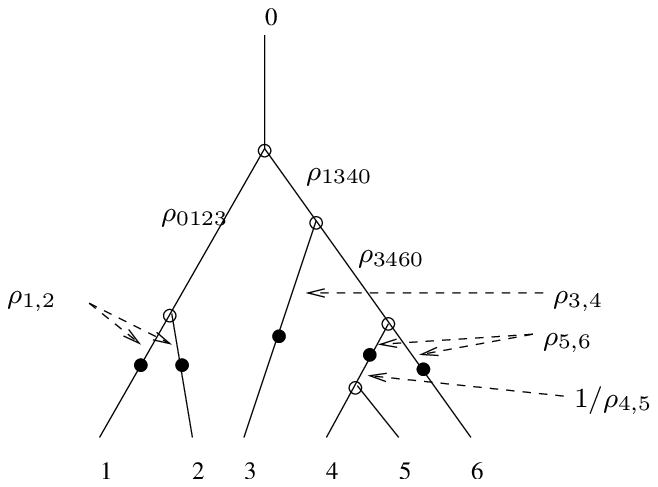} \ \ \includegraphics[height=1.5in]{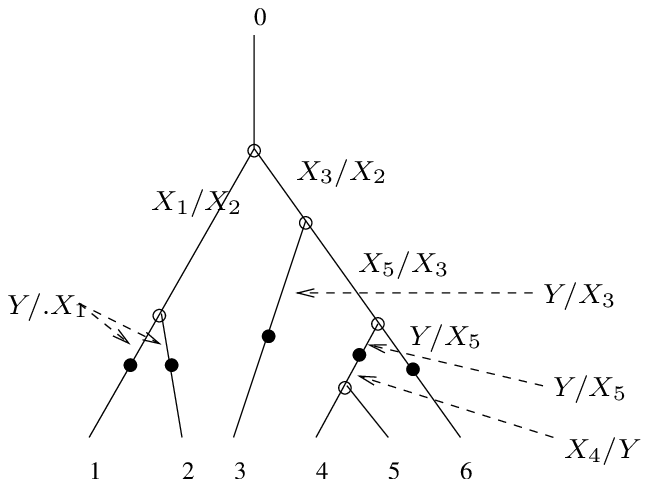}
\caption{A cross-ratio chart and a simple-ratio chart for the same maximal colored tree.}\label{example_charts}
\end{figure}

\begin{definition} 
Let $T$ be a colored tree.  A labelling $ \varphi: \Edge(T) \to
\R_{\ge 0}$ is {\em balanced} if it satisfies the following condition:
denote by $V_-(T)$ the set of vertices on the root side of the colored
vertices, that is, connected to the root by a path not crossing a
colored vertex. For each vertex $v_0 \in V(T)$ and any colored vertex
$v$ connected by a path of edges not crossing the root, let 
$\pi(v_0,v)$
denote the product of the values of $\varphi$ along the unique path of
edges from $v_0$ to $v$.  Then $\varphi$ is balanced if $\pi(v_0,v)$ is
independent of the choice of colored vertex $v$.  Let $X(T)$ denote
the set of balanced labellings:
$$
 X(T) := \{ \varphi:\Edge(T) \to \R_{\ge 0} \, | \, \forall v_0 \in
 V_-(T), \, 
\pi(v_0,v) \ \text{is independent of} \ v \in
 V_{\col}(T) \} .$$
%
\end{definition} 
We denote by $G(T) \subset X(T)$ the subset of non-zero labellings.

\begin{proposition} \label{maxhomeo}
Let $T$ be a maximal colored tree.  Then $\phi_T$ is a homeomorphism
from $\ol{M}_{n,1, \leq T}$ onto $X(T)$, mapping $M_{n,1}$ onto
$G(T)$.
\end{proposition}
\noindent Thus in particular the simple-ratios and cross-ratios define
the same topology on $\ol{M}_{n,1,\leq T}$.

\begin{proof}  It follows from the definition that 
$\phi_T$ takes values in balanced labellings, with products $y/x_i$
  where $i$ is the top vertex.  The construction of $\phi_T$ also makes it clear how to
  construct a pointed nodal quilted disk from a
  balanced labeling of $T$, showing that $\phi_T$ is onto
  $X(T)$. To make the relationship between the coordinates in the
  balanced labeling and the cross-ratio coordinates in a chart for
  $\ol{M}_{n,1,\leq T}$ explicit, let $\rho = (\rho_e)_{e\in
    \Edge(T)}$ be the cross-ratios in a chart covering
  $\ol{M}_{n,1,\leq T}$.  Without loss of generality, assume that all
  chart cross-ratios of the form $\rho_{ijkl}$ are either of the form
  $\rho_e = \rho_{ijk0}$, or $\rho_{e} = \rho_{0ijk}$, so that on
  $\ol{M}_{n,1, \leq T}$ they take values in $(-\infty, 0]$, and such
that for $[S] \in \ol{M}_{n,1,\leq T}$, $\rho_e(S) = 0$ implies that the
combinatorial type of $S$ has the edge $e$.  Let $\zeta =
(\phi_{T,e})_{e\in \Edge(T)}$ denote the simple ratios in a balanced
labelling of $T$.  We claim that $\rho_{e} = \zeta_{e}f(\zeta)$ for
every edge $e \in \Edge(T)$, where $f(\zeta) $ is a smooth function on
the interior of $\ol{M}_{n,1,\leq T}$ which is continuous up to the
boundary, and $f(\zeta) \neq 0$ on $\ol{M}_{n,1,\leq T}$.  In
particular, $\rho_{e} = 0 \iff \zeta_{e} = 0$.  First we prove it for
the cross-ratios $\rho_{ijk0}$ in the chart.  By symmetry it suffices
to consider the edge pictured in Figure \ref{equichart}, where an edge
$e$ joins vertices $v_r$ and $v_s$, with $v_r$ above $v_s$, so
$\phi_{T,e} = x_s/x_r$, and a chart cross-ratio for $e$ is $\rho_{ijk0}$.
\begin{figure}[h]
\includegraphics[height=1.5in]{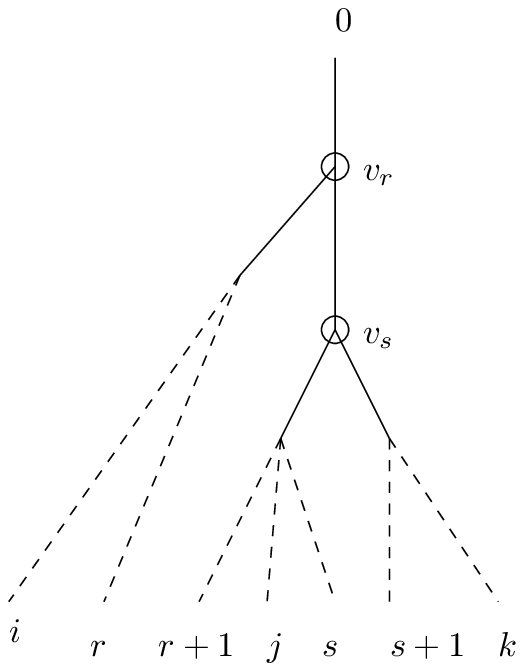}\quad \quad \includegraphics[height=1.5in]{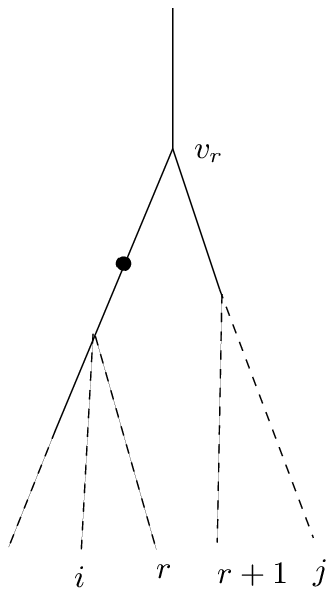}
\caption{Comparing cross-ratios with simple ratios.}\label{equichart}\label{equi_chart2}
\end{figure}
Parametrizing so that $z_0=\infty$, \bea \rho_{ijk0} & = & - \frac{z_j
  - z_k}{z_j - z_i} = - \frac{x_s}{x_{r}}\left( \frac{x_j/x_s +
  x_{j+1}/x_s + \ldots + 1 + \ldots + x_{k-1}/x_s}{x_i/x_r + \ldots +
  1 + \ldots + x_{j-1}/x_r } \right) = \phi_{T,e} f(\zeta).  \eea The
ratios in the bracketed function are products of ratios labeling edges
below $v_r$ and $v_s$, and the presence of the 1's means that
bracketed function is smooth and never zero for all positive non-zero
ratios and continuous as the labels in the chart go to 0, so
$\rho_{ijk0} = 0$ if and only if $\phi_{T,e} = x_s/x_r = 0$.  Now we
prove it for a cross-ratio $\rho_{ij}$ in the chart. Parametrizing so
that $z_0 = \infty$ and using $y$ to denote the height of the line
with respect to this parametrization, consider an edge such as the one
pictured in Figure \ref{equi_chart2}, where the cross-ratio labelling $e$ in
a cross-ratio chart is $\rho_{ij}$, and the simple ratio $\phi_{T,e} =
y/x_r$. Then
$$ \rho_{ij}  =  \frac{y}{z_j - z_i} = 
\frac{y}{x_r}\left(\frac{1}{\frac{x_i}{x_r} + \ldots + 1 + \ldots +
  \frac{x_{j-1}}{x_r}}\right) = \phi_{T,e} f(\zeta)$$ 
where the ratios appearing in the big bracket are products of simple
ratios labelling edges below $v_r$.  The function $f(z)$ is smooth and
never $0$ for all positive non-zero ratios and it is continuous as
ratios $\phi_{T,e} \to 0$.  Moreover $\rho_{ij} = 0$ if and only if
$y/x_r =0$. The case of a colored vertex above a regular vertex is
very similar so we omit it.  This proves that the transition from a
simple ratio chart to a cross ratios chart is a smooth change of
coordinates on $\ol{M}_{n,1,\leq T}$.
\end{proof}

One sees from this description that $\ol{M}_{n,1}$ is {\em not} a
manifold-with-corners.  We say that a point $[S] \in \ol{M}_{n,1}$ is
a {\em singularity} if $\ol{M}_{n,1}$ is not CW-isomorphic to a
manifold with corners near $[S]$.

\begin{example} 
The first singular point occurs for $n = 4$.  The expression
$(f(x_1)f( x_2)) (f(x_3) f(x_4))$ is adjacent to the expressions
{\small
$ f(x_1x_2)(f(x_3)f(x_4))$, $(f(x_1)f(x_2)) f(x_3x_4) $,
$f(x_1)(f(x_2)(f(x_3)f(x_4)))$, $ ((f(x_1)f(x_2))f(x_3))f(x_4) $}
and hence there are four edges coming out of the corresponding vertex.
On the other hand, the dimension of $M_{4,1}$ is $3$, see Figure \ref{singularity}.  Thus $M_{4,1}$
cannot be a manifold with corners (and therefore, not a simplicial
polytope.)  
\begin{figure}[h]
\includegraphics[height=2in]{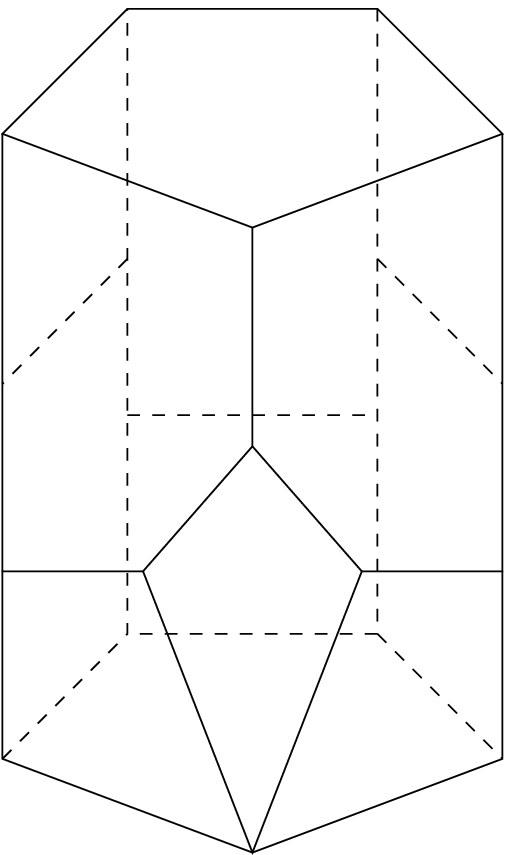}
\includegraphics[height=1.5in]{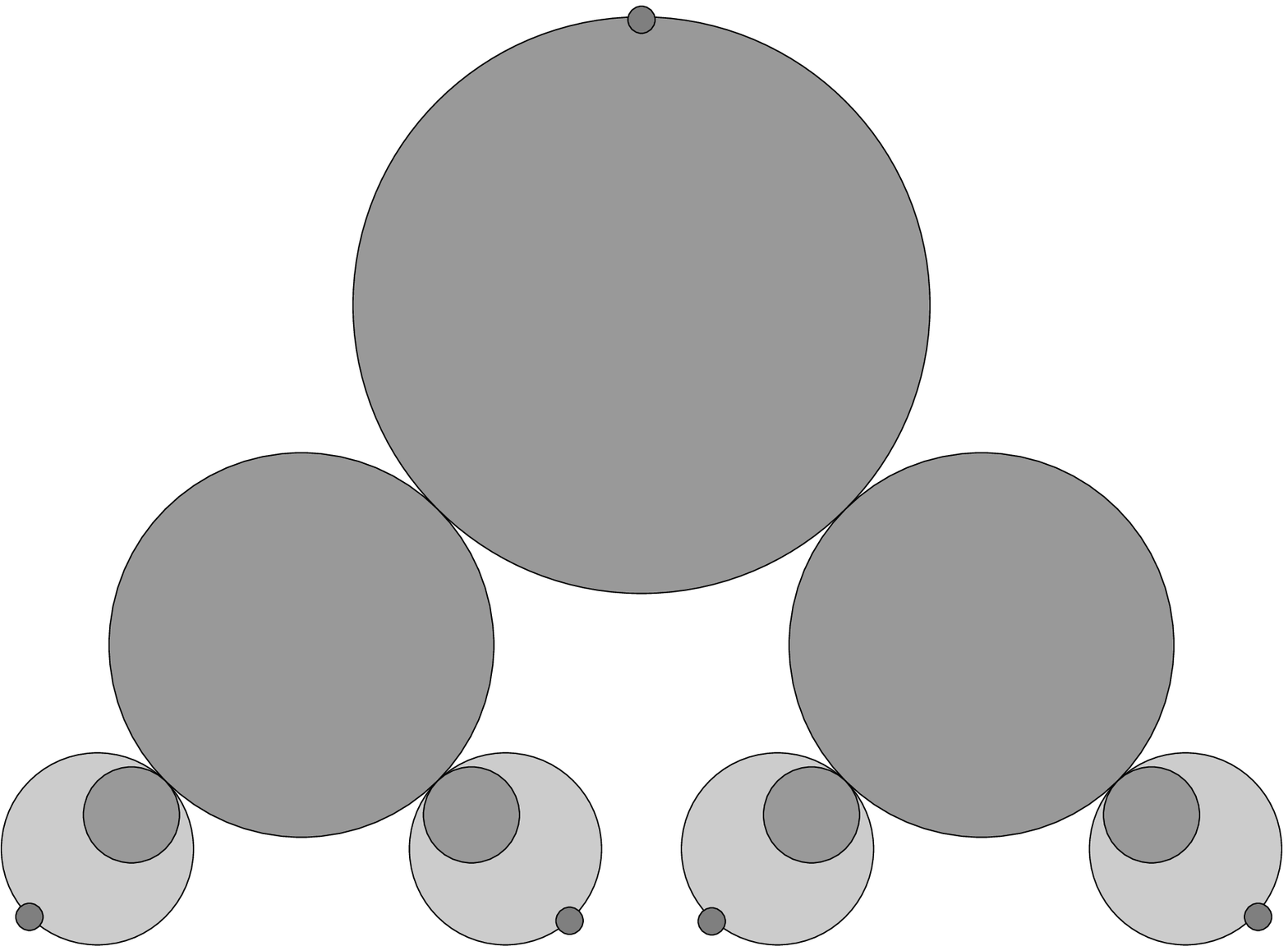}
\caption{$\ol{M}_{4,1}$, sometimes called the ``Chinese lantern''. The singular point on the boundary, which has 4 edges coming out of it, corresponds to the nodal quilted disk at right.}.\label{singularity}
\end{figure}
\end{example}

\begin{lemma} \label{morphism}  Any morphism of trees $f: T_0 \to T_1$ induces
a morphism of balanced labellings $X(f): X(T_0) \to X(T_1)$.
\end{lemma}

\begin{proof} 
Let $f: T_0 \to T_1$ be a morphism of trees.  Given a balanced
labelling $\varphi^m : \Edge(T_0) \to \R_{\ge 0 }$ we obtain a
balanced labelling $\varphi_1 : \Edge(T_1) \to \R_{\ge 0}$ by setting
$\varphi_1(e_1) = \prod \varphi_0(e_0)$ where the product is over
edges $e_0$ above $e_1$ that are collapsed under $f$.  One sees easily
that the resulting labelling of $T_1$ is balanced.
\end{proof}  

\begin{corollary}  \label{nonmaxhomeo} Let $T$ be a colored tree.  There 
exists a CW-isomorphism of $M_{n,1,T} \times X(T)$ onto a neighborhood
of $M_{n,1,T}$ in $\ol{M}_{n,1}$.
\end{corollary}  

\begin{proof}  Let $T^m$ be a maximal tree such that there exists
a morphism of trees $f: T^m \to T$.  For each vertex $v \in V(T)$, let
$T^m_v \subset T$ be the subtree whose vertices map to $v$.  Given a
labelling $\varphi_m \in X(T^m)$, we obtain by restriction labellings
$\varphi_v \in X(T^m_v)$ for each $v \in V(T)$, and from Lemma
\ref{morphism} a labelling $\varphi \in X(T)$.  Thus we obtain a map 
$$ X(T^m) \to  \left( \prod_{v \in V(T)} X(T^m_v) \right) \times X(T) .$$
It is straightforward to verify that this map induces an isomorphism
of $ \{ \varphi \in X(T^m) | \varphi(e) \neq 0 \forall e \in \Edge(T)
\}$ onto $(\prod_{v \in V(T)} X(T^m_v)^*) \times X(T) .$ The former is
the image of $\ol{M}_{n,1,\leq T}$ under $\phi_T^{-1}$.  Since each tree
$T^m_v$ is maximal, Proposition \ref{maxhomeo} gives an isomorphism of
$\ol{M}_{n,1,\leq T}$ onto $M_{n,1,T} \times X(T).$
\end{proof}  

For later use, we describe subsets of the edges whose labels determine
all others.

\begin{definition}
Let $T$ be a maximal colored tree. Let $e$ be an interior edge of
$T$ that is incident to a pair of trivalent vertices. Contraction of
$e$ produces a 4-valent vertex, and we say that the tree obtained by a
{\em flop}\label{flop_nef} of $e$ is that which corresponds to the
alternative resolution of the 4-valent vertex.  A {\em fusion} move
through an interior vertex $v_i$ is one by which two colored vertices
immediately below $v_i$ become a single colored vertex immediately
above $v_i$; we call the reverse move a {\em splitting} move.  We say
that two maximal colored trees $T$ and $T^\prime$ differ by a {\em
  basic move} if they differ by a single flop, fusion, or splitting
move.
\end{definition}

\begin{figure}[h]
\includegraphics[height=1in]{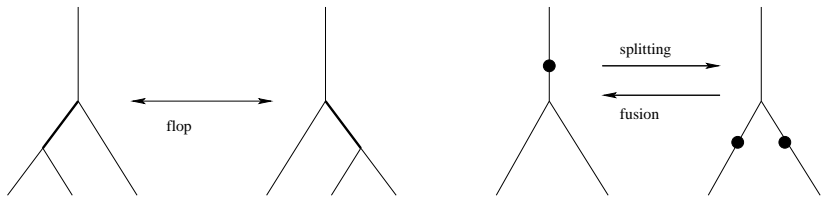}
\caption{Basic moves on edges in a colored
  tree.}\label{flop2_type}
\end{figure}

\noindent Any maximal colored tree can be obtained from a fixed
maximal colored tree by a sequence of basic moves.
Let $T$ be a maximal colored tree.  The simple ratios chart covering
the open set $\ol{M}_{n,1,\leq T}$ assigns a simple ratio coordinate to each
interior edge of $T$; however, on a given chart, some of those ratios
may be functions of other ratios in the chart. There are six
possibilities for an interior edge $e$ of $T$, pictured in Figure
\ref{assoc_moves}.  Note that the edges in the top four cases have a
basic move associated to them, but not the two lower cases.
\begin{figure}
\begin{center}
\includegraphics[height=1.5in]{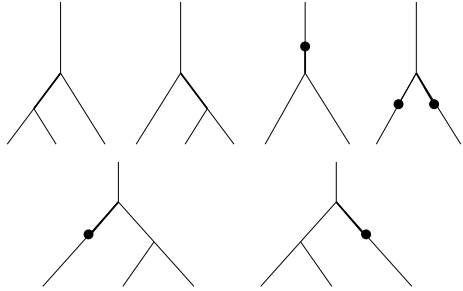}
\end{center}
\caption{The top 4 types of edge have associated basic moves; the lower two types do not.}\label{assoc_moves}
\end{figure}

\begin{lemma}
All simple ratios in $\phi_T$ are determined by the simple ratios
labeling edges which have an associated basic move.
\end{lemma}

\begin{proof}
The simple ratios for the other edges are redundant: if an edge $e$
doesn't have an associated basic move, it must be one of the lower two
types in Figure \ref{assoc_moves}. Let $v$ be the vertex directly
above $e$.  Observe that there must exist a path from $v$ to another
colored vertex below it, such that every edge in the path is one of
the top four types in Figure \ref{assoc_moves}.  Thus the simple
ratios labeling the edges in that path appear in the reduced chart,
and the relations imply that the product of the simple ratios in that
path is equal to the simple ratio labeling $e$.
\end{proof}
\begin{definition}\label{reduced_chart}
A {\em reduced simple ratio chart} is given by restricting $\phi_T$ to
the edges that have an associated basic move; we call these edges {\em
  reduced chart edges}.
\end{definition}

\begin{example}
In the example of Figure \ref{example_charts}, the reduced simple ratio chart consists of $y/x_1, x_1/x_2, x_3/x_2, x_5/x_3, y/x_5, x_4/y$.  The simple ratio $y/x_3$ is redundant.  %
\end{example}

\section{Colored metric ribbon trees}

The multiplihedra have another geometric realization as {\em colored
  metric ribbon trees}.  
Colored trees were introduced in
Boardman-Vogt \cite{boardman-vogt}, although their construction does
not have the relations described below.

\begin{definition} 
A {\em colored rooted metric ribbon tree} is a colored rooted ribbon
tree with a metric $\lambda: E(T) \to (0,\infty)$ such that the sum of
the edge lengths in any non-self-crossing path from a colored vertex
$v \in V_{\col}(T)$ back to the root is independent of $v \in
V_{\col}(T)$.  A colored rooted metric ribbon tree is {\em stable} if
each colored resp. non-colored vertex has valency at least $2$
resp. $3$.
\end{definition}

\begin{example} 
For the tree in Figure \ref{bicolor}, an edge length map is subject to
the relations
$ \lambda_1 + \lambda_2 + \lambda_3 = \lambda_1 + \lambda_2 +
  \lambda_4 = \lambda_1 + \lambda_2 + \lambda_5 = \lambda_1 +
  \lambda_6 = \lambda_7.  $
\begin{figure}[h]
\includegraphics[height=2in]{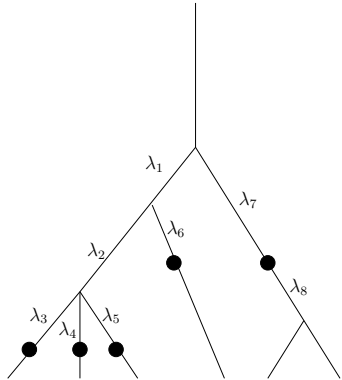}
\caption{A colored ribbon tree.  The relations on $\lambda_1, \ldots, \lambda_8$ imply that $\lambda_3 = \lambda_4 = \lambda_5, \lambda_3 + \lambda_2 = \lambda_6,$ and $\lambda_6 + \lambda_1 = \lambda_7$.}
\label{bicolor}
\end{figure}
\end{example}

For each stable colored tree $T$, we denote by $\WW_{n,1,T}$ the set
of all maps $\lambda$ colored metric trees with underlying colored
tree $\lambda$ and by $\WW_n$ the union
$$ \WW_{n,1} = \bigcup\limits_{T} \WW_{n,1,T}.$$
We define a topology on $\WW_{n,1,T}$ as follows.  Assume that a
sequence $T_i = (T,\{\lambda_i\}_{i\in \N})$ of metric trees converges
for each interior edge $e$ to a non-negative real number.  In other
words, $\lambda_i(e) \to \lambda_\infty(e) \in [0, \infty)$ for every
  $e \in E(T)$.
We say that the corresponding colored metric trees $T_i$ converge to
$T_\infty$ if
\begin{enumerate}
\item $T_\infty$ is the tree obtained from $T$ by collapsing edges $e$
  for which $\lambda_\infty(e) := \lim_{i \to \infty} \lambda_i(e)
  =0$.  This defines a surjective morphism of colored rooted ribbon
  trees, $f: T \to T_\infty$.
\item $V_{\col}(T_\infty) = f(V_\col(T)) $
\item $\lambda_\infty(e) = \lim_{i \to \infty} \lambda_i(f^{-1}(e))$,
if this limit is non-zero. 
\end{enumerate}

\begin{proposition}
$\WW_{n,1,T}$ is a polyhedral cone in $\R^n$, where 
$n = |E(T)| - |V_\col(T)| + 1$.  
\end{proposition}
\begin{proof}
There is an $\R_+$ action on $\WW_{n,1,T}$, given by $(\delta \cdot
\lambda)(e):= \delta \lambda(e)$, so it is clearly a cone.  The
dimension follows from the fact that there are $|E(T)|$
variables and $|V_\col(T)|-1$ relations.  The polyhedral structure can be
seen by writing $|V_\col(T)| - 1$ variables as linear combinations of $n$
independent variables.  Then the condition that all $\lambda(e) \ge 0$
means that $\WW_{n,T}$ is an intersection of half-spaces.
\end{proof}

\begin{example}
In the example of Figure \ref{bicolor}, $|E(T)| = 8$, and $|V_\col(T)| = 5$. 
We can choose independent variables to be $\lambda_1, \lambda_2, \lambda_3, \lambda_8$, and express the remaining variables as
$ \lambda_4  =  \lambda_3, \ \lambda_5  =  \lambda_3, \ \lambda_6
 =  \lambda_2 + \lambda_3, \ \lambda_7  =  \lambda_1 + \lambda_2 +
\lambda_3.  $
Thus the space of admissible edge lengths is parametrized by points in
the polyhedral cone that is the intersection of $\R_+^4$ (for the
independent variables being non-negative) with the half-spaces
$\lambda_4 \ge 0, \lambda_5 \ge 0, \lambda_6 \geq 0$ and
$\lambda_7\geq 0$.

\end{example}
   
Exponentiating the labellings of the edges gives a map 
$$\Theta: \WW_{n,1} \to M_{n, 1}, \quad (\Theta(\lambda)(e) =
e^{-\lambda(e)}$$
Since $\lambda \geq 0$, the image of a cone $\WW_{n,1,T}$ is
identified directly with the subset of $\ol{M}_{n,1,\leq T}$
consisting of balanced labellings with $\phi_{T,e} \in (0,1]$ for every
$e \in \Edge(T)$.

\begin{example}
Consider the case $n=3$, where we have fixed the parametrization of
the elements of $M_{n,1}$ so that the interior circle is identified
with a line of height $L$ in half-space.  Let $x = z_2-z_1$ and $y =
z_3 - z_2$.  The images of $\WW_{3,1,T}$ subdivide $\R_{>0}^2$ into 6
regions, see Figure \ref{conesmultiplihedron}, each of which
corresponds to a cone in $\R^2$ via the homeomorphism $(x,y) \mapsto
(\log x, \log y)$.
  
\begin{figure}[h]
\includegraphics[height=3in]{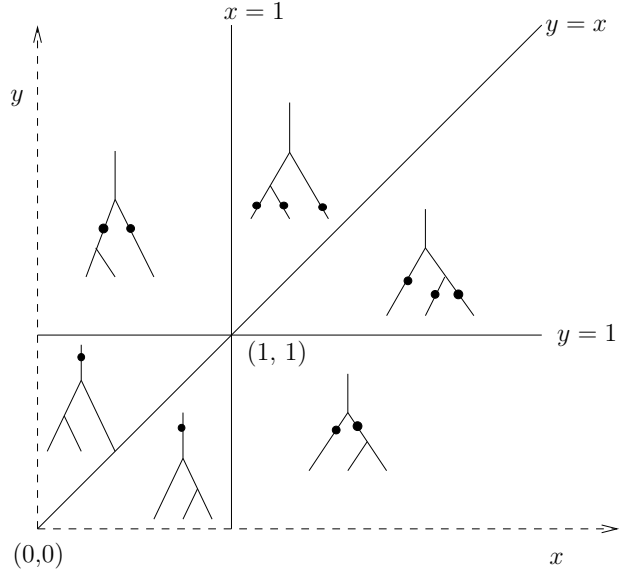}
\caption{The image of the cones of $\WW_{3,1}$ in the moduli space $M_{3,1}$.} 
\label{conesmultiplihedron}
\end{figure}

\end{example}

There is a natural compactification $\ol{\WW}_{n,1}$ of $\WW_{n,1}$ by
allowing edges to have length $\infty$. The map $\Theta$ extends to
the compactifications by taking limits in appropriate charts.

\begin{theorem} \label{homeo3} The map 
$\Theta: \ol{\WW}_{n,1} \to \ol{M}_{n,1}$ is a homeomorphism, with the
  property that for any combinatorial type ${T}$,
  $\Theta(\ol{\WW}_{n,1,T})$ intersects $M_{n,{T},1}$ in a single
  point.
\end{theorem}
\noindent This is the colored analog of Theorem \ref{homeo2}.

\begin{proof}
As $\lambda(e) \to \infty$, the identification $\lambda(e) \to \phi_{T,e} = e^{-\lambda}$
implies $\phi_{T,e} \to 0$. Thus the image of a compactified
cone 
$\ol{\WW}_{n,1,T}$ is identified with the subset of $\ol{M}_{n,1,\leq
  T}$ consisting of balanced labellings with $\phi_{T,e} \in [0,1]$ for
every $e \in E(T)$.
\end{proof}

\section{Toric varieties and moment polytopes.}

In this section we show
\begin{theorem}  \label{thmtoric} $\ol{M}_{n,1}$ is homeomorphic to the non-negative
part of an embedded toric variety $V$ in $\P^k(\C)$, where $k$ is the
number of maximal colored trees with $n$ leaves.
\end{theorem} 
\noindent In particular, $\ol{M}_{n,1}$ is isomorphic as a
$CW$-complex to a convex polytope; this reproduces the result of
Forcey \cite{forcey-2007}.  Using this we prove 
the main Theorem \ref{main}. 
First we define the toric variety $V$. 
Recall from Section \ref{simple_charts_2} that a point in $M_{n,1}$ can be
identified with a projective coordinate
$ {\ul{x}} = (x_1: x_2: \ldots: x_{n-1}: y),$
by parametrizing such that $z_0 = \infty$, and setting $x_i = z_{i+1}
- z_i$ and $y$ to be the height of the line.  Let $T$ be a maximal
colored tree.  Adapting the algorithm of Forcey in \cite{forcey-2007},
we associate a {\em weight vector} $\mu_T \in \Z^{n}$ to the tree
$T$ as follows. A pair of adjacent leaves in $T$, say $i$ and $i+1$,
determines a unique vertex in $T$, which we label $v_i$. Let $a_i$ be
the number of leaves on the left side of $v_i$, and let $b_i$ be the
number of leaves on the right side of $v_i$.  Let
\[
\delta_{i} = 
\left\{ 
\begin{array}{cl}
0 &  \mbox{if $v_i$ is below the level of the colored vertices, and}\\
1& \mbox{if $v_i$ is above the colored vertices}.
\end{array}
\right.
\]
Set
\[
\mu_T := (a_1 b_1(1 + \delta_1) , \ldots, a_i b_i(1+\delta_i), \ldots, a_{n-1} b_{n-1}(1+\delta_{n-1}), - \sum\limits_i \delta_i a_i b_i).
\]   
\begin{example}
The tree in Figure \ref{exampleweight2} has weight vector $(2, 16, 6, 1, 4,-14)$, and monomial $x_1^2 x_2^{16} x_3^6 x_4 x_5^4 y^{-14} $.
\begin{figure}[h]
\includegraphics[height=2.5in]{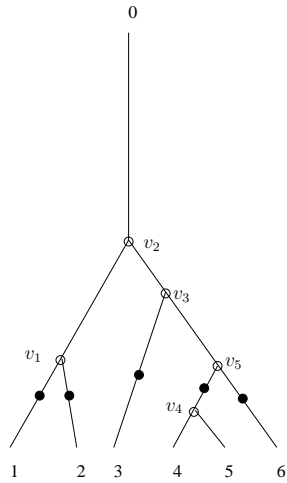}
\caption{A maximal colored tree, whose weight vector is $(2, 16, 6, 1, 4,-14)$.}\label{exampleweight2}
\end{figure}
\end{example}

Fix some ordering of the $k$ maximal colored trees with $n$ leaves,
$T_1, \ldots, T_k$. The projective toric variety $V \subset
\P^{k-1}(\C)$ is the closure of the image of the embedding
\begin{equation}\label{embedding}
( x_1 : \ldots : x_{n-1}: y ) \mapsto ( {\ul{x}}^{\mu_{T_1}} : \ldots:
  {\ul{x}}^{\mu_{T_k}}).
\end{equation}
The entries in the weight vectors always sum to $n(n-1)/2$, so the map is well-defined on the homogeneous coordinates.

\begin{lemma}\label{flop2_lemma}
Suppose that two maximal colored trees $T$ and $T^\prime$ differ by
a single basic move involving an edge $e \in \Edge(T)$.  Let $\phi_{T,e}$ denote
the simple ratio labeling the edge $e$ in the chart determined by $T$.
Then
\[
\frac{{\ul{x}}^{\mu_{T^\prime}}}{{\ul{x}}^{\mu_{T}}} = \phi_{T,e}^m
\]
for an integer $m>0$.  In general, for two maximal trees $T$ and $T^\prime$,
\[
\frac{{\ul{x}}^{\mu_{T^\prime}}}{{\ul{x}}^{\mu_{T}}} = \phi_{T,e_1}^{m_1} \phi_{T,e_2}^{m_2} \ldots \phi_{T,e_r}^{m_r}
\]
for some edges $e_1, \ldots, e_r$ of $T$ and positive integers $m_1, \ldots, m_r$.
\end{lemma}

\begin{proof}
  
First let us consider the case of a single
flop. Without loss of generality consider the situation in Figure
\ref{flop_tree}.  Say $T$ is on the left, and $T^\prime$ is on the
right, and the affected edges are in bold.
\begin{figure}[h]
\includegraphics[height=2in]{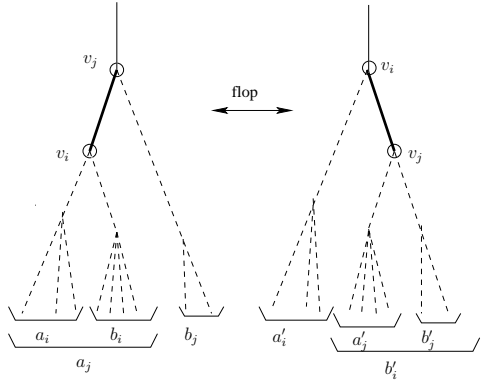}
\caption{The effect of a flop on weight vectors.}\label{flop_tree}
\end{figure}
The weight vectors $\mu_T$ and $\mu_{T^\prime}$ are the same in all entries except entries $i$ and $j$, where
$$
(\mu_T)_i  =  a_i b_i, \quad 
(\mu_T)_j  =  a_j b_j = (a_i + b_i) b_j, \quad (\mu_{T^\prime})_i = a_i (b_i + b_j),  \quad 
 (\mu_{T^\prime})_j  = b_i b_j.
$$
Therefore,
$$ 
\frac{{\ul{x}}^{\mu_{T^\prime}}}{{\ul{x}}^{\mu_{T}}}  =  \frac{x_i^{a_i (b_i + b_j)} x_j^{b_i b_j}}{x_i^{a_i b_i} x_j^{(a_i + b_i) b_j}}
 =  \frac{x_i^{a_i b_j}}{x_j^{a_i b_j}}
 =  \left( \frac{x_i}{x_j} \right)^{a_i b_j}
$$
and observe that $\phi_{T,e} = x_i/x_j$ is the ratio labeling that edge of $T$, and $a_ib_j \geq 1$.

For the other kinds of basic move it suffices to consider fusion, in which a pair of colored vertices are below $v_i$ in $T$, and above $v_i$ in $T^\prime$.  The weight vectors
of $T$ and $T^\prime$ are identical in all entries except for the
$i$-th entry, which corresponds to the exponent of $x_i$, and the
$n+1$-th entry, which corresponds to the exponent of $y$:
$$
(\mu_T)_i  =  2 a_i b_i, \quad 
\mu_{T^\prime}  =  a_i b_i, \quad 
(\mu_{T^\prime})_{n+1} - (\mu_{T})_{n+1}  =  -(0) - (- a_i b_i)
$$
Therefore 
$$
\frac{{\ul{x}}^{\mu_{T^\prime}}}{\mu_{T}}  =  \frac{x_i^{a_i b_i}y^{-0}}{x_i^{2 a_i b_i} y^{-a_i b_i}}
                                      =  \frac{y^{a_i b_i}}{x_i^{a_i b_i}}
                                      =  \left( \frac{y}{x_i}\right)^{a_i b_i}$$
where $\phi_{T,e} = y/x_i$ is the ratio labeling the two edges below $v_i$ of
$T$, and $a_ib_j \geq 1$.

The vertices are partially ordered by their positions in the tree; the
effect of basic moves on the partial ordering are individual changes $(v_i
\leq v_j) \leftrightarrow (v_j \leq v_i)$, or $(v_\col \leq v_i) \leftrightarrow (v_i \leq v_\col)$, between adjacent vertices.  In general, every maximal tree is
obtained from a {\em fixed} tree $T$ by a sequence of independent basic 
moves -- by independent we just mean that each one involves a
different pair of vertices.  We prove the general case by induction on the number of
independent basic moves needed to get from a fixed maximal tree $T$, to any
other maximal tree $T^\prime$.  Having proved the base case,
now consider a tree $T^\prime$ obtained after a sequence of $k+1$
flops. Write $\widetilde{T}$ for a tree which is $k$ independent moves
away from $T$ and one move away from $T^\prime$.  Suppose that the the
final move between $\widetilde{T}$ and $T^\prime$ is described by the
$(v_i \leq v_j) \to (v_j \leq v_i)$. By the inductive 
hypothesis and the base step, 
$$ \frac{{\ul{x}}^{\mu_{T^\prime}}}{{\ul{x}}^{\mu_{T}}}
 =  \frac{{\ul{x}}^{\mu_{T^\prime}}}{{\ul{x}}^{\mu_{\widetilde{T}}}} \frac{{\ul{x}}
^{\mu_{\widetilde{T}}}}{{\ul{x}}^{\mu_T}}  = 
\left(\frac{x_i}{x_j} \right)^m \phi_{T,e_1}^{m_1} \phi_{T,e_2}^{m_2} \ldots \phi_{T,e_r}^{m_r}$$ 
for some positive integers $m_1, \ldots, m_r$ and $m$, and some edges
$e_1, \ldots, e_r$ of $T$. Since none of the
previous flops involved the pair $v_i$ and $v_j$, the partial order in
the original tree $T$ must have also had $v_i \leq v_j$, although they
were possibly not adjacent in $T$.  In any case, the ratio $x_i/x_j$
is a product of the ratios in the chart $\phi_T$ labeling the edges from $v_i$ to $v_j$.
The case where the final move is one of $(v_\col, v_i) \leftrightarrow (v_i,v_\col)$ is 
similarly straightforward.  This completes the inductive step.
 \end{proof}

\begin{proof}[Proof of Theorem \ref{thmtoric}] 

We use Lemma \ref{flop2_lemma} to identify the simple ratios in a
reduced chart for each maximal colored tree $T_i$, with the
non-negative part of the affine slice $V \cap \mathbb{A}_i$.  Consider
$T_1$.  The affine piece $V \cap \mathbb{A}_1$ consists of all points
\[
\left( 1: \frac{{\ul{x}}^{\mu_{T_2}}}{{\ul{x}}^{\mu_{T_1}}}: \ldots : \frac{{\ul{x}}^{\mu_{T_k}}}{{\ul{x}}^{\mu_{T_1}}} \right)
\] 
where the entries may be 0. Now let $\phi_{T_1,e_1}, \ldots, \phi_{T_1,e_l}$ be the simple
ratio coordinates in a reduced chart for the open set $M_{n,1,\leq T_1} \subset
\ol{M}_{n,1}$ (Definition \ref{reduced_chart}).  By construction, the
edges $e_1,\ldots, e_l$ of $T_1$ have associated
basic moves. Thus they determine a set $T_1(e_1), \ldots, T_1(e_l)$ of
maximal colored trees, where each 
$T_1(e_i)$ is obtained from $T_1$ by the basic move associated to
the edge $e_i$.  Without loss of generality, assume that
the $l$ maximal trees $T_2, \ldots, T_{l+1}$ are respectively
$T_1(e_1), \ldots, T_1(e_l)$. By Lemma \ref{flop2_lemma}, we
identify the non-negative part of $V \cap \mathbb{A}_1$ with the chart
$\phi_{T_1}$ by the map
$$ \ol{M}_{n,1,T_1}  \lra  V\cap \mathbb{A}_1 \quad  (\phi_{T_1,e_1}, \ldots, \phi_{T_1,e_l})  \mapsto
 (1: \phi_{T_1,e_1}^{m_1} : \phi_{T_1,e_2}^{m_2}: \ldots : \phi_{T_1,e_l}^{m_l}: * : \ldots :*) $$
where $m_1, \ldots, m_l$ are positive integers depending on the
combinatorics of $T_1$, and the entries labeled $*$ are higher products of
$\phi_{T_1,e_1}, \ldots, \phi_{T_1,e_l}l$.  This map is well-defined, one-to-one and onto for
$\phi_{T_1,e_i}$ and $\phi_{T_1,e_i}^{m_i}$ which are all in the non-negative range $ [0,\infty)$.
\end{proof}

\begin{corollary} $\ol{M}_{n,1}$ is $CW$-isomorphic to the convex hull of the weight vectors in $\R^n$, and thus $CW$-isomorphic to a $(n-1)$-dimensional polytope. 
\end{corollary}

\begin{proof}
The non-negative part of a projective toric variety constructed with
weight vectors is homeomorphic, via the moment map, to the convex hull
of the weight vectors (see, for example, \cite{fulton-intro},
\cite{sottile-2002}).
\end{proof}

\begin{proof}[Proof of Theorem \ref{main}]
By induction on $n$: The one-dimensional spaces $\ol{M}_{2,1},
\ol{M}_3, J_{2}$ and $K_3$ are all compact and connected, and so
$CW$-isomorphic.  It suffices, therefore, to show that $\ol{M}_{n,1}$
is the cone on its boundary. This is true since it is homeomorphic to
a convex polytope.
\end{proof}

\section{Stable weighted disks}

Fukaya, Oh, Ohta, and Ono \cite{fooo} introduced another geometric
realization of the multiplihedron, although the CW-structure is
slightly different.  
A {\em weighted stable $n+1$-marked disk} consists of
\begin{enumerate}
\item a stable nodal disk $(\ul{\Sigma} = (\Sigma_1,\ldots, \Sigma_m),
  \ul{z} = (z_0,\ldots, z_n))$
\item  for each component $\Sigma_1,\ldots, \Sigma_m$ of $\ul{\Sigma}$,
a {\em weight} $x_i \in [0,1]$
\end{enumerate}
with the following property: if $\Sigma_i$ is further away from
$\Sigma_j$ from the {\em root marking} $z_0$ then $x_i \leq
x_j$.  An {\em isomorphism} of rooted stable disks is an
isomorphism of stable disks intertwining with the weights.
Let $\ol{M}^w_{n+1}$ denote the moduli space of stable weighted marked
disks, equipped with the natural extension of the Gromov topology in
which a sequence $(\ul{\Sigma}_i, \ul{z}_i, \ul{\lambda}_i)$ converges
to $(\ul{\Sigma},\ul{z},\ul{\lambda})$ if $(\ul{\Sigma}_i, \ul{z}_i)$
Gromov converges to $(\ul{\Sigma},\ul{z})$ and the weights on the
limit curve are pulled back from those on $\ul{\Sigma}$ via the
morphism of trees appearing in the limit.  

For example, $\ol{M}^w_{3}$ is an interval; $\ol{M}^w_{4}$ is a
hexagon consisting of a square and two triangles, joined along two
edges, see Figure \ref{weighted4}.  Each triangle is defined by the
inequality $0 \leq x_2 \leq x_1 \leq 1$.  The moduli space
$\ol{M}^w_5$ has $23$ cells of dimension $2$ on the boundary ($2$
projecting onto $2$-cell of $\ol{M}_5$, $10$ projecting onto $1$-cells
of $\ol{M}_5$, and $11$ projecting onto vertices of $\ol{M}_5$.)  On
the other hand, the multiplihedron $\ol{M}_{4,1}$ has $13$ cells of
dimension $2$ on the boundary, see Figure \ref{singularity}.

\begin{figure}[ht]
\begin{picture}(0,0)%
\includegraphics{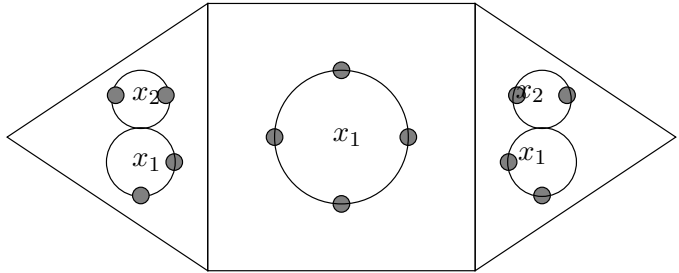}%
\end{picture}%
\setlength{\unitlength}{3947sp}%
\begingroup\makeatletter\ifx\SetFigFont\undefined%
\gdef\SetFigFont#1#2#3#4#5{%
  \reset@font\fontsize{#1}{#2pt}%
  \fontfamily{#3}\fontseries{#4}\fontshape{#5}%
  \selectfont}%
\fi\endgroup%
\begin{picture}(4224,1704)(1189,-2653)
\put(3248,-1853){\makebox(0,0)[lb]{{{$x_1$}%
}}}
\put(1988,-2011){\makebox(0,0)[lb]{{{$x_1$}%
}}}
\put(1988,-1591){\makebox(0,0)[lb]{{{$x_2$}%
}}}
\put(4578,-1579){\makebox(0,0)[rb]{{{$x_2$}%
}}}
\put(4595,-1970){\makebox(0,0)[rb]{{{$x_1$}%
}}}
\end{picture}%

\caption{Moduli of weighted $4$-marked disks}
\label{weighted4}
\end{figure}

\section{Stable scaled affine lines.}\label{complex_multi}

In this section we re-interpret the moduli space of quilted disks as a
{\em moduli space of stable scaled lines}.  This construction has the
advantage that it works for any field.  Working over $k =\C$ gives a
moduli space introduced Ziltener's study \cite{zilt:phd} of gauged
pseudoholomorphic maps from the complex plane; we show it is a
projective variety with toric singularities.

\begin{definition} Let $k$ be a field.  
A {\em scaled marked line} is a datum $(\bA,\ul{z},\phi)$, where $\bA$
is an affine line over $k$, $\ul{z} = (z_1,\ldots,z_n) \in \bA$ are
distinct points, and $\phi \in \Omega^1(\bA,k)^k$ is a translationally
invariant area form.  An {\em isomorphism of scaled marked lines} is
an isomorphism $\psi: \bA \to \bA$ that intertwines the area forms and
markings.  A scaled line is {\em stable} if the automorphism group is
finite, that is, has at least one marking.  Denote by $M_{n,1}(k)$ the
corresponding moduli space of stable, scaled marked lines.
\end{definition}

In the case $k = \R$, $M_{n,1}(\R)$ has as a component (given by
requiring that the markings appear in order) the moduli space
$M_{n,1}$ of the previous section, through identifying $(z_1, \ldots,
z_n, \phi)$ with $(z_0 = \infty, z_1, \ldots, z_n, L_\phi)$, where
$L_\phi \subset \bA$ is a line of height $1/\phi(1)$.  $M_{n,1}(k)$
has a natural compactification, obtained by allowing the points to
come together and the volume form $\phi$ to scale.  For any nodal
curve $\ol{C}$ with markings $z_0,\ldots, z_n$, and component
$\ol{C}_\alpha$ of $\ol{C}$, we write ${z}_{\alpha i}$ for the special
point in $\ol{C}_\alpha$ that is either the marking $z_i$, or the node
closest to $z_i$.

\begin{definition} A {\em (genus zero) nodal scaled
marked line} is a datum $(\ol{C},z,\phi)$, where
  $\ol{C}$ is a (genus zero) projective nodal curve, $\ul{z} =
  (z_0,\ldots,z_n)$ is a collection of markings disjoint from the
  nodes, and for each component $\ol{C}_\alpha$ of $\ol{C}$,
  the affine line $C_\alpha := \ol{C}_\alpha \setminus\{ {z}_{\alpha 0}\}$ is
  equipped with a (possibly zero or infinite) translationally
  invariant volume form $\phi_i \in \Omega^1(C_\alpha,k)^k$.  We call a
  volume form $\phi_i$ {\em degenerate} if it is zero or infinite.  An
  automorphism of a stable nodal scaled curve is an automorphism of
  the nodal curve preserving the volume forms and the markings.  A
  nodal scaled marked curve is {\em stable} if it has finite
  automorphism group, or equivalently, if each component with
  non-degenerate (resp. degenerate) volume form has at least two
  (resp. three) special points.
\end{definition}

The affine structure on $C_\alpha$ is unique up to dilation, so that
$\Omega^1(C_\alpha,k)^k$ is well-defined.  The {\em combinatorial type} of
a nodal scaled marked affine line is a rooted colored tree: 
Vertices represent components of the nodal curve, edges represent
nodes, labeled semi-infinite edges represent the markings, with the
root always labelled by $z_0$.  Every path from a leaf back to the
root must pass through exactly one colored vertex.

Now we specialize to the case $k = \C$.  $M_{n,1}(\C)$ contains as a
subspace those scaled marked curves such that all markings lie on the
projective real line, $\R P:=\R \cup \{\infty\}$; these are naturally
identified with marked disks.  More accurately, $M_{n,1}(\C)$ admits
an antiholomorphic involution induced by the antiholomorphic
involution of $\P^1(\C)$.  The involution extends to an
antiholomorphic involution of $\ol{M}_{n,1}(\C)$. The multiplihedron
$\ol{M}_{n,1}$ can be identified with the subset of the fixed point
set such that the points are in the required order.

We introduce coordinates on $M_{n,1}(\C)$ in the same way as we did
for $M_{n,1}$.  Define two types of coordinates, of the form
$\rho_{ijkl}$ where $i, j, k, l$ are distinct indices in $0, 1,
\ldots, n$, and of the form $\rho_{ij}$, where $i,j$ are distinct
indices in $1, \ldots, n$. The $\rho_{ijkl}$ are defined as before,
and the $\rho_{ij}$ are defined as follows: given a representative
$(z_1, z_2, \ldots, z_n, \phi)$, 
$$ \rho_{ij}([z_1, \ldots, z_n, \phi]) := (\phi(1)(z_j -
  z_i))^{-1}.$$
The coordinates extend to the compactification $\ol{M}_{n,1}(\C)$. For
the coordinates $\rho_{ijkl}$, we evaluate the cross-ratio at a component $\ol{C}_\alpha$ in the bubble tree
for which at least three of $z_{\alpha i}, z_{\alpha j}, z_{\alpha k}$
and $z_{\alpha l}$ are distinct, normalizing by 
\begin{equation}\label{norm}
\rho_{ijkl} = \left\{ \begin{array}{ccccc} \infty, & \mbox{if} & z_{\alpha i} =
  z_{\alpha j} & \mbox{or} & z_{\alpha k} = z_{\alpha l},\\ 1, & \mbox{if} & z_{\alpha i} = z_{\alpha k} &
  \mbox{or} & z_{\alpha j} = z_{\alpha l},\\ 0, & \mbox{if} & z_{\alpha i} = z_{\alpha l} & \mbox{or} &
  z_{\alpha j} = z_{\alpha k}.
\end{array}\right.
\end{equation}

For the coordinates $\rho_{ij}$, we evaluate them at the unique
component $\ol{C}_\alpha$ at which $z_{\alpha 0}, z_{\alpha i}$ and
$z_{\alpha j}$ are distinct, normalizing by
\[
\rho_{ij} = \left\{ \begin{array}{ll} 0, & \mbox{if $C_\alpha$ has
    infinite scaling,}\\ \infty, & \mbox{if $C_\alpha$ has zero scaling.}
\end{array}
\right.
\]  
The same arguments as in the real case show that the product of
forgetful morphisms defines an embedding
$$ \rho_{n,1}: \ol{M}_{n,1} \to (\P^1(\C))^{(n+1)n(n-1)(n-2)/4! +
  n(n-1)/2} .$$
The coordinates $\rho_{ijkl}$ and $\rho_{ij}$ also satisfy recursion relations
$$
\rho_{jklm} = \frac{\rho_{ijkm}-1}{\rho_{ijkm}-\rho_{ijkl}}\label{recursion},
\quad 
\rho_{jk} = \frac{\rho_{ij}}{\rho_{ijk0}}\label{relation}.
$$
Let $\ol{A}_{n,1}(\C)$ denote the closure of the algebraic variety
defined by the two types of cross ratio coordinate and the relations
(\ref{recursion}).

 \begin{theorem}\label{bijection}
 The map $\rho_{n,1}: \ol{M}_{n,1}(\C) \to \ol{A}_{n,1}(\C)$ is a bijection.
 \end{theorem} 

The proof of the bijection is an extension of the corresponding result
for genus zero stable nodal $(n+1)$-pointed curves,
$\ol{{M}}_{n+1}(\C)$. In this case the cross-ratios $\rho_{ijkl}$
satisfy (\ref{recursion}). Let $\ol{A}_n(\C)$ denote the closure of
the algebraic variety defined by the cross-ratio coordinates and the
relation (\ref{recursion}). The image of the canonical embedding
$\rho_n(\ol{{M}}_{n+1}(\C))$ with cross-ratios is contained in
$\ol{A}_n(\C)$, the
map $\rho_n : \ol{{M}}_{n+1}(\C) \to \ol{A}_n(\C)$ is a bijection
\cite[Theorem D.4.5]{mcd-sal}.

\begin{proof}[Proof of Theorem \ref{bijection}]

First we show that $\rho_{n,1}$ is injective. Given a nodal stable
scaled marked line $(\ol{C},\ul{z},\phi)$, by construction the
combinatorial type uniquely determines which cross-ratios
$\rho_{ijkl}$ are $0,1$ or $\infty$, and which cross-ratios
$\rho_{ij}$ are $0$ and $\infty$.  In addition, the isomorphism class
of each component of $\ol{C}$ is determined by the cross-ratios
$\rho_{ijkl}$ with values in $\P^1(\C) \setminus \{0,1, \infty\}$ and
$\rho_{ij}$ with values in $\P^1(\C) \setminus\{0,\infty\}$, so the
map $\rho_{n,1}$ is injective.  To show that $\rho_{n,1}$ is
surjective, let $\mathbf{\rho} \in \ol{A}_{n,1}(\C)$. By the result
for stable curves \cite[D.4.5]{mcd-sal}, there is a unique stable,
nodal $(n+1)$-marked genus zero curve of combinatorial type given by a
rooted tree $T$ (the root corresponds to the marking $z_0$), which
realizes the cross-ratios of the form $\rho_{ijkl}$.

\begin{lemma}\label{refine_lemma}
Let $\alpha \in V(T)$, and suppose that for some $1\leq i <j \leq n$,
$z_{\alpha_i}, z_{\alpha_j}, z_{\alpha_0} $ are distinct at $\alpha$.
\begin{enumerate}
\item If $\rho_{ij} = 0$, then for every vertex $\beta \in V(T)$ in a
  path from $\alpha$ to the root (including $\alpha$ itself),
  $\rho_{kl} = 0$ for every distinct triple $z_{\beta_k},
  z_{\beta_l}, z_{\beta_0}$.
\item If $\rho_{ij}=\infty$, then for every vertex $\beta \in V(T)$
  in a path from $\alpha$ away from the root (including $\alpha$
  itself), $\rho_{kl}=\infty$ for every distinct triple $z_{\beta_k},
  z_{\beta_l}, z_{\beta_0}$.
\item If $0<|\rho_{ij}| < \infty$, then
\begin{enumerate}

\item for every other distinct triple $z_{\alpha_k}, z_{\alpha_l},
  z_{\alpha_0}$ on $\alpha$, $\rho_{kl} \notin \{0, \infty\}$;

\item for every vertex $\beta \in V(T)$ adjacent to $\alpha$ towards the root, and every distinct triple $z_{\beta_k}, z_{\beta_l}, z_{\beta_0}$,  $\rho_{kl} = 0$;

\item for every vertex $\beta \in V(T)$ adjacent to $\alpha$ away from
  the root, and every distinct triple $z_{\beta_k}, z_{\beta_l},
  z_{\beta_0}$, $\rho_{kl} = \infty$.
\end{enumerate}

\end{enumerate}
\end{lemma}

\begin{proof}
(a) First, we show that $\rho_{kl} = 0$ for every $k,l$ such that
  $z_{\alpha_k}, z_{\alpha_l}, z_{\alpha_0}$ are distinct.  Without
  loss of generality suppose that $z_{\alpha_k}$ is distinct from
  $z_{\alpha_j}$ and $z_{\alpha_0}$.  Then $\rho_{ijk0} \neq 0$ and so
  by (\ref{relation}), $\rho_{jk} = \rho_{ij}/\rho_{ijk0} = 0$.  Now
  without loss of generality suppose that $z_{\alpha_l}$ is distinct
  from $z_{\alpha_k}$ and $z_{\alpha_0}$.  Then $\rho_{jkl0} \neq 0$
  so by (\ref{relation}) $\rho_{kl} = \rho_{jk}/\rho_{jkl0} = 0/
  \rho_{jkl0} = 0$.  Now consider the vertex $\beta$ that is
  immediately adjacent to $\alpha$ in the direction of the root.
  Without loss of generality, suppose that $z_{\beta_k}$ is distinct
  from $ z_{\beta_i} = z_{\beta_j}$.  The combinatorics of $T$ at
  $\alpha$ and $\beta$ imply that $\rho_{ijk0} = \infty$, hence by
  (\ref{relation}), $\rho_{jk} = \frac{\rho_{ij}}{ \rho_{ijk0}} =
  0/\infty = 0$.  Applying the first argument that $\rho_{kl} = 0$
  for all $k,l$ with $z_{\beta_k}, z_{\beta_l}, z_{\beta_0}$ distinct.
  The result holds by remaining vertices in the path from $\alpha$ by
  induction.
(b)
First, we show that $\rho_{kl}=\infty$ for every $k,l$ such that $z_{\alpha_k}, z_{\alpha_l}, z_{\alpha_0}$ are distinct.  Without loss of generality suppose that $z_{\alpha_k}$ is distinct from $z_{\alpha_j}$ and $z_{\alpha_0}$. Then $\rho_{ijk0}\neq \infty$, hence by (\ref{relation}) $\rho_{jk} = \rho_{ij}/\rho_{ijk0} = \infty/\rho_{ijk0} = \infty$.  Now without loss of generality suppose that $z_{\alpha_l}$ is distinct from $z_{\alpha_k}$ and $z_{\alpha_0}$.  Then $\rho_{jkl0} \neq \infty$ hence by (\ref{relation}) $\rho_{kl} = \rho_{jk}/\rho_{jkl0} = \infty/\rho_{jkl0} = \infty$.  Now consider a vertex $\beta$ that is immediately adjacent to $\alpha$ away from the root.  It is now enough to show that $\rho_{mn}=\infty$ for some $m,n$ such that $z_{\beta_m}, z_{\beta_n}$ and $z_{\beta_0}$ are distinct. Pick $k$ and $l$ such that $z_{\alpha_k}, z_{\alpha_l}, z_{\alpha_0}$ are distinct (hence by the previous argument $\rho_{kl} = \infty$), and such that $\beta$ is adjacent to $\alpha$ through a node that identifies $z_{\alpha_l}$ with $z_{\beta_0}$. Now let $z_{\beta_m}$ be distinct from $z_{\beta_l}$ and $z_{\beta_0}$. Then $\rho_{klm0} \neq \infty$, so by (\ref{relation}), $\rho_{lm} =\rho_{kl}/\rho_{klm0} = \infty/\rho_{klm0} = \infty$.

(c) {\it Proof of (i):} If $z_{\alpha_k}$ is distinct from
$z_{\alpha_i}, z_{\alpha_j}$ and $z_{\alpha_0}$, then $\rho_{ijk0}
\notin \{0,1,\infty\}$ so $\rho_{jk} = \rho_{ij}/\rho_{ijk0}$ hence
$0 < |\rho_{jk}| < \infty$. Repeating this argument implies that $0<
|\rho_{kl}| < \infty$ for any $k$ and $l$ such that $z_{\alpha_k},
z_{\alpha_l}, z_{\alpha_0}$ are distinct.  {\it Proof of (ii):} In
light of (a) and the proof of (c)(i), it is enough to prove that for
any $k$ such that $z_{\beta_j}, z_{\beta_k}$ and $z_{\beta_0}$ are
distinct, then $\rho_{jk} = 0$. Note that since $\beta$ is closer to
the root than $\alpha$, $z_{\beta_i} = z_{\beta_j}$. Hence,
$\rho_{ijk0} = \infty$, and by (\ref{relation}), $\rho_{jk} =
\rho_{ij}/\rho_{ijk0} = 0$.  {\it Proof of (iii):} In light of (b)
and (c)(i), it is enough to prove the following case: if $\alpha$ is
incident to $\beta$ in such a way that $z_{\beta_i}=z_{\beta_0}$ is
distinct from $z_{\beta_j}$ and $z_{\beta_k}$, then $\rho_{jk} =
\infty$. In this case, $\rho_{ijk0} = 0$, so (\ref{relation}) implies
$\rho_{jk} = \rho_{ij}/\rho_{ijk0} = \infty$.

\end{proof}

By Lemma \ref{refine_lemma}, the vertices of the tree $T$ can be
partitioned into subsets for which the cross-ratios $\rho_{ij}$
defined on them are 0, $\infty$, or finite non-zero. Let
\begin{eqnarray*}
V_0 & := & \{ \alpha \in V(T) \lvert z_{\alpha_i}, z_{\alpha_j},
z_{\alpha_0} \mbox{ are distinct and } \ \rho_{ij} = 0\},\\
V_f & := & \{ \alpha \in V(T) \lvert z_{\alpha_i}, z_{\alpha_j},
z_{\alpha_0} \mbox{ are distinct and } \ 0< |\rho_{ij} | < \infty\},\\
V_\infty & := & \{ \alpha \in V(T) \lvert z_{\alpha_i}, z_{\alpha_j},
z_{\alpha_0} \mbox{ are distinct and } \ \rho_{ij} =\infty\}.
\end{eqnarray*}
If $V_0$ is empty, turn the marked point $z_0$ into a nodal point and
attach it to the nodal point $\zeta$ of a scaled curve
$(\ol{C}^\prime,z_0, \zeta, \phi)$.  If $V_0$ is non-empty, by Lemma
\ref{refine_lemma} it must be a connected sub-tree which includes the
component containing the root $z_0$.  If a marked point $z_i$, $i=1,
\ldots, n$ is on a component labeled by $\alpha \in V_0$, turn the
marked point $z_i$ into a nodal point $z_{\alpha_i}$ and attach it to
the nodal point $\zeta$ of a scaled curve $(\ol{C},\zeta, z_i, v)$.
If $\alpha \in V_f$ then by Lemma \ref{refine_lemma} it is attached by
a node to $V_0$.  Suppose that $z_{\alpha_i}, z_{\alpha_j}$ and
$z_{\alpha_0}$ are distinct and $0 < |\rho_{ij} | <\infty$.  Identify
this sphere component and its markings with a stable marked curve with
the same markings, and a volume form $\phi$ determined by
parametrizing $z_{\alpha_0} = \infty, z_{\alpha_i} = 0,
z_{\alpha_j}=1$ and putting $1/\phi(1) = \rho_{ij}$.  Finally, suppose
that $\alpha \in V_{\infty}$ is connected by a nodal point
$z_{\alpha_0}$ to a nodal point $z_{\beta_i}$ of $V_0$.  Then insert a
stable marked curve $(\ol{C}, \zeta_0, \zeta_1, \phi)$ such that the
node identifications are $\zeta_0$ with $z_{\beta_i}$, and $\zeta_1$
with $z_{\alpha_0}$.

At the end of this process one obtains a stable nodal, marked scaled curve in $\ol{M}_{n,1}(\C)$, whose  combinatorial type is a colored tree refining the tree $T$, and whose image under the cross-ratio embedding is the same as the original point $\mathbf{\rho} \in \ol{A}_{n,1}(\C)$.
\end{proof}

Given the combinatorial type of a stable scaled curve, one can choose
a local chart of cross-ratios according to the same prescription as
given in the real case.

\begin{figure}[h]
\includegraphics[height=2in]{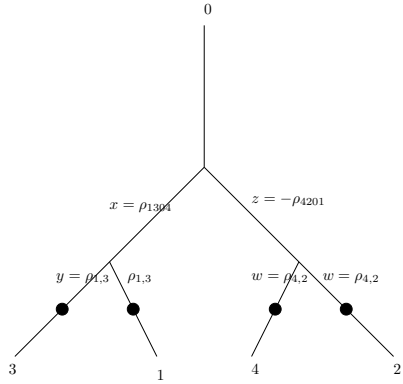}
\caption{A cross-ratio chart in $\ol{M}_{4,1}(\C)$.  The relation is $xy = zw$, which gives a toric singularity.} 
\label{complex_chart}
\end{figure}

The subset $G({T}) = X({T}) \cap \Hom(\Edge({T}),\C^*) $ of points
with non-zero labels is the kernel of the homomorphism
$ \Hom(\Edge({T}),\C^*) \to \Hom(\Vert^-({T}),\C^*) $
given by taking the product of labels from the given vertex to the
colored vertex above it, and is therefore an algebraic torus.  The
torus $G({T})$ acts on $X({T})$ by multiplication with a dense orbit.
Choose a planar structure on $T$.  Let
$$ \phi_T: \ol{M}_{n,1,\leq T}(\C) \to X(T) $$
denote the map given by the simple ratios in Definition
\eqref{simpleratios} (now allowed to be complex).  After re-labelling
it suffices to consider the case that the ordering is the standard
ordering.  We denote by $X^*(T)$ the Zariski open subset of $X(T)$
defined by the equations $1 + x_{i+1}/x_i + \ldots + x_{j}/x_i = (z_i
- z_j)/(z_i - z_{i+1}) = 0$, for $1 \leq i < j \leq n$.

\begin{theorem}  \label{toricprop} 
$\phi_T$ is an isomorphism of $\ol{M}_{n,1,\leq T}(\C)$ onto $X^*(T)$.
\end{theorem}    

\begin{proof}  Let $T$ be a maximal colored tree 
and consider the map $ \ol{M}_{n,1,\leq T} \to X(T)$ given by the
simple ratios.  The same argument as in the real case shows that any
$\lambda \in X(T)$ is in the image of some quilted disk unless at some
stage the reconstruction procedure assigns the same position to two
markings $i,j$ in different branches; in this case we have $\lambda
\in X^*(T)$.  The set of exceptional points in $X(T)$ is an affine
subvariety of $X(T)$ disjoint from $0 \in X(T)$ hence the Theorem.
\end{proof}

\begin{corollary}  Let $T$ be a colored tree.  
There exists an isomorphism of a Zariski open neighborhood of
$M_{n,1,T}(\C) \times \{0 \}$ in $M_{n,1,{T}}(\C) \times X({T})$ with
a Zariski open neighborhood of $M_{n,1,{T}}(\C)$ in
$\ol{M}_{n,1}(\C)$.  Thus $\ol{M}_{n,1}(\C)$ is a projective variety
with at most toric singularities.
\end{corollary} 

The proof is similar to the real case in Corollary \ref{nonmaxhomeo}
and left to the reader.  This completes the proof of Theorem 1.2 in
the introduction.

\addtocontents{toc}{\protect\enlargethispage*{1000pt}}


\begin{thebibliography}{10}
\addtocontents{toc}{\protect\pagebreak}

\def\cprime{$'$}
\bibitem{boardman-vogt}
J.~M. Boardman and R.~M. Vogt.
\newblock {\em Homotopy invariant algebraic structures on topological spaces}.
\newblock Springer-Verlag, Berlin, 1973.
\newblock Lecture Notes in Mathematics, Vol. 347.

\bibitem{forcey-2007}
Stefan Forcey.
\newblock Convex hull realizations of the multiplihedra. arXiv:math.AT/0706.3226.
\bibitem{fooo} K.~Fukaya, Y.-G Oh, H.~Ohta and K.Ono.  \newblock
  Lagrangian intersection {F}loer theory: anomaly and obstruction. Book in
  preparation.

\bibitem{zero-loop}
Kenji Fukaya and Yong-Geun Oh.
\newblock Zero-loop open strings in the cotangent bundle and {M}orse homotopy.
\newblock {\em Asian J. Math.}, 1(1):96--180, 1997.

\bibitem{fulton-intro}
William Fulton.
\newblock {\em Introduction to toric varieties}, volume 131 of {\em Annals of
  Mathematics Studies}.
\newblock Princeton University Press, Princeton, NJ, 1993.
\newblock , The William H. Roever Lectures in Geometry.

\bibitem{iwase-mimura}
Norio Iwase and Mamoru Mimura.
\newblock Higher homotopy associativity.
\newblock In {\em Algebraic topology (Arcata, CA, 1986)}, volume 1370 of {\em
  Lecture Notes in Math.}, pages 193--220. Springer, Berlin, 1989.

\bibitem{kontsevich-manin}
M.~Kontsevich and Yu. Manin.
\newblock Gromov-{W}itten classes, quantum cohomology, and enumerative geometry
  [ {MR}1291244 (95i:14049)].
\newblock In {\em Mirror symmetry, II}, volume~1 of {\em AMS/IP Stud. Adv.
  Math.}, pages 607--653. Amer. Math. Soc., Providence, RI, 1997.

\bibitem{mcd-sal}
Dusa McDuff and Dietmar Salamon.
\newblock {\em {$J$}-holomorphic curves and symplectic topology}, volume~52 of
  {\em American Math. Soc. Colloq. Pub.}.
\newblock American Mathematical Society, Providence, RI, 2004.

\bibitem{morphism} 
Khoa Nguyen and Chris Woodward.
\newblock Morphisms of cohomological field theories.  2008 preprint.  

\bibitem{sottile-2002}
Frank Sottile.
\newblock Toric ideals, real toric varieties, and the moment map.
\newblock In {\em Topics in algebraic geometry and geometric modeling}, volume
  334 of {\em Contemp. Math.}, pages 225--240. Amer. Math. Soc., Providence,
  RI, 2003.

\bibitem{stasheffbook}
James Stasheff.
\newblock {\em {$H$}-spaces from a homotopy point of view}.
\newblock Lecture Notes in Mathematics, Vol. 161. Springer-Verlag, Berlin,
  1970.

\bibitem{stasheff}
James~Dillon Stasheff.
\newblock Homotopy associativity of {$H$}-spaces. {I}, {II}.
\newblock {\em Trans. Amer. Math. Soc. 108 (1963), 275-292; ibid.},
  108:293--312, 1963.

\bibitem{zilt:phd}
F.~Ziltener.
\newblock {\em Symplectic vortices on the complex plane and quantum
  cohomology}.
\newblock PhD thesis, Zurich, 2006.

\end{thebibliography}
\end{document}